\documentclass[11pt]{amsart}

\usepackage{amsmath, amssymb, amsthm}
\usepackage[hidelinks]{hyperref}
\usepackage[parfill]{parskip}        
\usepackage[margin=1.05in]{geometry} 
\usepackage{tikz}
\usetikzlibrary{decorations.markings}
\usepackage{tikz-cd}
\usetikzlibrary{matrix}
\usepackage{comment}
\usepackage{mathtools}

\title{Self-similarity and limit spaces of substitution tiling semigroups}\keywords{} \subjclass{}

\author[James Walton]{James J.\ Walton}
\address{School of Mathematical Sciences, The University of Nottingham, University Park, Nottingham, NG7 2RD, United Kingdom}
\email{James.Walton@nottingham.ac.uk}
 
\author[Michael F.\ Whittaker]{Michael F.\ Whittaker}
\address{School of Mathematics and Statistics, University of Glasgow, University Place, Glasgow Q12 8QQ, United Kingdom}
\email{Mike.Whittaker@glasgow.ac.uk}

\def\R{\mathbb{R}}
\def\N{\mathbb{N}}
\def\Z{\mathbb{Z}}

\def\OO{\mathcal{O}}

\def\NN{\mathcal{N}}

\def\PP{\mathcal{P}}

\def\ep{\varepsilon}

\newcommand{\Om}{\Omega}
\newcommand{\Omp}{\Omega_\mathrm{punc}}
\newcommand{\Rp}{R_\mathrm{punc}}

\newcommand{\mc}{\mathcal}

\newcommand{\dom}{\mathrm{dom}}
\newcommand{\ran}{\mathrm{ran}}
\newcommand{\sub}{\varphi}

\newcommand{\supp}{\mathrm{supp}}
\newcommand{\aeq}{\sim_\mathrm{ae}}
\newcommand{\limsp}{\mathcal{J}}

\newtheorem{theorem}{Theorem}[section]
\newtheorem{proposition}[theorem]{Proposition}
\newtheorem{lemma}[theorem]{Lemma}

\theoremstyle{definition}
\newtheorem{definition}[theorem]{Definition}
\newtheorem{example}[theorem]{Example}

\newtheorem{remark}[theorem]{Remark}

\numberwithin{equation}{section} 
\numberwithin{figure}{section}   

\tikzstyle{vertex}=[circle]
\tikzstyle{goto}=[->,shorten >=1pt,>=stealth,semithick]

\newcommand{\HHex}[4]{
\begin{scope}[xshift=#1cm,yshift=#2cm,rotate=#3,scale=0.5^(#4)]
\draw (0:1) -- (60:1) -- (120:1) -- (180:1) -- (0:1); 
\end{scope}}
\newcommand{\HHexI}[4]{
\begin{scope}[xshift=#1cm,yshift=#2cm,rotate=#3,scale=0.5^(#4)]
\HHex{0}{0}{0}{1} 
\HHex{0}{0.866}{180}{1} 
\HHex{0.75}{0.433}{120}{1} 
\HHex{-0.75}{0.433}{240}{1} 
\end{scope}}
\newcommand{\HHexII}[4]{
\begin{scope}[xshift=#1cm,yshift=#2cm,rotate=#3,scale=0.5^(#4)]
\HHexI{0}{0}{0}{1} 
\HHexI{0}{0.866}{180}{1} 
\HHexI{0.75}{0.433}{120}{1} 
\HHexI{-0.75}{0.433}{240}{1} 
\end{scope}}

\makeatletter
\@namedef{subjclassname@2020}{%
  \textup{2020} Mathematics Subject Classification}
\makeatother

\thanks{This research was partially supported by EPSRC grant EP/R013691/1.}
\keywords{aperiodic tilings; self-similar; semigroups; tiling dynamics}
\subjclass[2020]{Primary: 37B52; Secondary: 20M18; 52C22}

\begin{document}
\maketitle

\centerline{\em Dedicated to our late colleague and friend Uwe Grimm}

\begin{abstract} We show that Kellendonk's tiling semigroup of an FLC substitution tiling is self-similar, in the sense of Bartholdi, Grigorchuk and Nekrashevych. We extend the notion of the limit space of a self-similar group to the setting of self-similar semigroups, and show that it is homeomorphic to the Anderson--Putnam complex for such substitution tilings, with natural self-map induced by the substitution. Thus, the inverse limit of the limit space, given by the limit solenoid of the self-similar semigroup, is homeomorphic to the translational hull of the tiling.
\end{abstract}

\section{Introduction}

We study aperiodic tilings of Euclidean space arising from a substitution rule. Aperiodic tilings provide important examples of topological dynamical systems, over spaces of tilings called \emph{tiling spaces} \cite{Sadun}. Since the seminal work of Anderson and Putnam \cite{AP}, much is known about the topology of tiling spaces of substitution tilings, including calculations of their topological invariants \cite{BDHS,Sadun}. A beautiful result of Sadun and Williams \cite{SW} proves that the tiling space of an FLC tiling is a fibre bundle over the torus, with totally disconnected fibres (which are homeomorphic to the Cantor set, in the case of a repetitive, nonperiodic tiling). Kellendonk \cite{Kel1,Kel2} established a complementary algebraic approach, by constructing an inverse semigroup of partial translations for such a tiling. Tiling semigroups were extensively studied by Kellendonk and Lawson in \cite{KL} and their algebraic properties were determined by Zhu in \cite{Zhu}. In this paper, we show that tiling semigroups arising from a substitution are self-similar, and that the limit solenoid naturally associated to the tiling semigroup is homeomorphic to the tiling space.

Self-similarity of groups has been a hugely fruitful mechanism with which to construct groups enjoying certain properties, particularly growth properties. As such they have been key in solving several important problems in Group Theory and Geometric Group Theory, most notably Grigorchuk's famous group with intermediate growth \cite{nek_book, Gr, Gr2}. Self-similar inverse semigroups acting on topological Markov shifts were introduced by Bartholdi, Grigorchuk and Nekrashevych in \cite{BGN} and Nekrashevych went on to show they give rise to Smale spaces in \cite{Nek_SSIS}. It has already been noted that self-similar groups and semigroups can be associated to some substitution tilings. However, the general theory for substitution tilings is not worked out in either paper, which focus entirely on a quotient by rotations and reflections of the Penrose tilings, see \cite[p.13--15]{BGN} and \cite[p.859--861]{Nek_SSIS}. Moreover, we outline a different, more global approach, which makes it implicit that the full object constructed is Kellendonk's tiling semigroup, rather than starting with generators and (self-similar) relations. Our constructions are easily modified to construct the analogous object where translations are replaced with general rigid motions, but we consider it important to present the theory in the translational case, which has remained a major focus in Aperiodic Order owing to connections between translational dynamics of aperiodic patterns and their spectral properties, which finds application to the study of quasicrystals \cite{BG}.

For a substitution tiling satisfying standard conditions we show that Kellendonk's tiling semigroup acts self-similarly on a topological Markov chain that is naturally homeomorphic to the canonical transversal of the tiling space. We recall the substitution graph, whose path-space is a topological Markov chain. It is well-known that this is conjugate to Kellendonk's discrete hull of a tiling with dynamics arising from (the inverse of) substitution. The tiling semigroup acts on this space by translation. Indeed, an element of the tiling semigroup specifies a patch of tiles with two distinguished tiles that specify the domain and range of the translation between them. The self-similarity arises from the fact that translation across patches of tiles can be lifted to translations between supertiles, with these structures being analogous at all levels of the hierarchy in the topological Markov chain.

In \cite{AP}, Anderson and Putnam define a branched manifold from a substitution tiling, now called the Anderson--Putnam (AP) complex. Starting with a collection of prototiles, the AP complex is the quotient space defined by the relation that two prototiles are glued along their codimension-1 faces if those faces ever meet in a tiling. For a substitution tiling, one can define addresses of points of prototiles by a left-infinite topological Markov shift. We define the limit space of a general self-similar semigroup action, in an analogous way to the self-similar group case \cite{BGN}, as a quotient space of such left-infinite topological Markov shifts by asymptotic equivalence, which itself we modify for the semigroup case (Definition \ref{def:ae}). We prove that the limit space of a tiling semigroup is homeomorphic to the AP complex. The limit space naturally inherits a self-map from the shift map of the Markov shift, which corresponds to the usual self-map by substitution on the AP complex, which thus has inverse limit homeomorphic to the tiling space.

The paper is organised as follows. In Section \ref{sec: Nonperiodic tilings and their semigroups} we recall the basic facts required about aperiodic tilings, their tiling semigroups, tiling spaces and tiling substitutions. In Section \ref{sec: Self-similarity of substitution tiling semigroups} we prove that tiling semigroups of substitution tilings (with standard restrictions) are self-similar. We give a general notion of a self-similar semigroup being contracting (Definition \ref{def: contracting}) and show this property holds for substitution tiling semigroups. In Section \ref{sec: limit space} we introduce the asymptotic equivalence relation for self-similar semigroups (Definition \ref{def:ae}). The associated quotient space, called the limit space, is shown to be homeomorphic to the Anderson--Putnam complex (Theorem \ref{thm: limit spaces of sub tilings}). Finally, in Section \ref{sec: examples}, we see how the definitions given apply to various examples.


\section{Nonperiodic tilings and their semigroups} \label{sec: Nonperiodic tilings and their semigroups}

Tilings will be built from a finite set $\PP$ of \textbf{prototiles}, labelled compact subsets of $\R^d$ that are equal to the closure of their interior. The support of $p$ is denoted $\supp(p) \subset \R^d$. A translate of a prototile is called a {\bf tile} and a finite, connected set of tiles which overlap only on their boundaries is called a {\bf patch}. By {\bf connected} here, we mean that any two tiles can be connected through a path of {\bf meeting tiles}. By tiles {\bf meeting} we allow two options: that the tiles are {\bf adjacent}, meaning that they intersect non-trivially (on their boundaries) or alternatively, if the tiles and patches have cell decompositions (for example, the tiles are polyhedra), then the tiles meet along a shared codimension-1 face. When cells have a cellular decomposition, the constructions to follow will hold for whichever of these two conventions the reader prefers. The set of all finite patches is denoted $\PP^*$. A \textbf{tiling} $T$ is a covering of $\R^d$ by tiles which intersect only on their boundaries. Given a tiling $T$ and bounded subset $S \subseteq \R^d$ we define
\[
T \sqcap S \coloneqq \{t \in T \mid \supp(t) \cap S \neq \varnothing \}.
\]
If $S$ is a closed ball of radius $r$ then $T \sqcap S$ is called an {\bf $r$-patch}. If $T$ is a tiling and $x \in \R^d$, the translate of $T$ by $x$ is $T+x \coloneqq \{t+x \mid t \in T\}$ and the \textbf{orbit} of $T$ is $\OO(T) \coloneqq \{T+x \mid x \in \R^d\}$. We say that $T$ is \textbf{nonperiodic} if $T+x=T$ implies that $x=0$.

\begin{definition}\label{tiling metric}
For tilings $T$, $T'$ we define their distance in the \textbf{tiling metric} as
\[
d(T,T') \coloneqq \inf \{\ep, 1 \mid (T-x) \sqcap B_{1/\ep} = (T'-x') \sqcap B_{1/\ep},\ x, x' \in \R^d,\ |x|, |x'| < \varepsilon\},
\]
where $B_r$ denotes the closed ball of radius $r$ centred at $0 \in \R$.
\end{definition}

Two tilings $T$, $T'$ are close if $T$ and $T'$ have the same patch of tiles on a large ball centred about the origin, up to a small translation. The \textbf{continuous hull} (or {\bf tiling space}) of a tiling $T$ is the space of tilings whose finite patches all belong to $T$, up to translation, with topology induced by the tiling metric (that is, it is the space of tilings which are {\bf locally indistinguishable} from $T$). Equivalently, $\Om_T$ may be regarded as the completion of $\OO(T)$ under the tiling metric. We call $T$ {\bf repetitive} if, for every finite patch $P$, there exists some $r > 0$ so that a translated copy of $P$ can be found in every $r$-patch of $T$. In this case, every element of $\Om_T$ has the same set of finite patches and thus $\Om_T = \Om_{T'}$ for all $T' \in \Om_T$. In particular, if $T$ is nonperiodic and repetitive, then $T$ is {\bf strongly aperiodic}, that is, every element of $\Om_T$ is nonperiodic. A tiling $T$ is said to have \textbf{finite local complexity (FLC)} if there are only a finite number of two-tile patches in $T$, up to translation (equivalently, there are finitely many $r$-patches for each $r > 0$). Of course, every repetitive tiling has FLC. With the topology above, FLC is equivalent to compactness of $\Omega_T$ \cite[Lemma 2]{RW}.

A \textbf{substitution} on a set of prototiles $\PP$ is a map $\sub \colon \PP \to \PP^*$ for which there is a scaling factor $\lambda>1$ with $\supp(\sub(p)) = \lambda \cdot \supp(p)$ for each $p \in \PP$. Since the support of a substituted tile is exactly equal (rather than just covering) its inflated tile, $\sub$ is more specifically a {\bf stone inflation}. This property is necessary in what follows, but the inflation being a similarity $x \mapsto \lambda x$ is not and can instead be taken as an expansive linear map; we assume an expansion constant merely for expository convenience. A substitution $\sub$ on a tile $t=p+x$ is defined to be $\sub(t) \coloneqq \sub(p) + \lambda x$. Then a substitution may be applied to a patch, which by a slight abuse of notation we also denote $\sub \colon \PP^* \to \PP^*$; similarly, substitution may be applied to tilings. An \text{$n$-supertile} is a translate of the patch $\sub^n(p)$ for some $p \in \PP$. We will always assume that $\sub$ generates FLC tilings, which is to say that it generates only finitely many two-tile patches (up to translation equivalence) under iteration. We call $\sub$ {\bf primitive} if there exists some $k \in \N$ so that, for each $a$, $b \in \PP$, we have that $\sub^k(a)$ contains a translated copy of $b$.

A tiling $T$ is {\bf admitted} by the substitution if every finite sub-patch of $T$ is contained in some $n$-supertile. The set of such tilings is denoted $\Omega_\sub$. It follows easily from FLC that $\Omega_\sub$ is non-empty and, if $\sub$ is primitive, then every admitted tiling is repetitive, so that $\Omega_\sub = \Omega_T$ for any $T \in \Omega_\sub$.

The induced map $\sub \colon \Omega_\sub \to \Omega_\sub$ is always surjective, so that for each tiling $T$ there is a corresponding `supertiling', which decomposes under $\sub$ to $T$. If $\sub$ is additionally injective, then we say that $\sub$ is {\bf recognisable}. This means that, for any $T \in \Omega_\sub$, there is a unique way to group tiles into supertiles, whose associated tiling in $\Omega_\sub$ decomposes to $T$ under substitution. By continuity and compactness, this may always be done by a locally defined rule for a recognisable substitution. Recognisability is equivalent to non-periodicity of the tilings of $\Omega_\sub$ \cite{Sol2}. We will assume throughout that $\sub$ is recognisable.

We always assume that $\sub$ {\bf forces the border} \cite[p.24]{Kel1}. This means that there is some $k \in \N$ so that any $k$-supertile $\sub^k(p)$, for $p \in \PP$, extends uniquely to a valid patch containing $\sub^k(p)$ and any tiles intersecting the boundary of $\sub^k(p)$ (that is, $\sub^k(p)$ uniquely extends to its $1$-corona). Border forcing may always be assumed by passing to a dynamically equivalent substitution by collaring tiles (see \cite{AP,Sadun}).

The \textbf{Anderson--Putnam complex} of a tiling is the compact Hausdorff topological space formed through taking the transitive closure of gluing together prototiles in all ways their translations can be adjacent in a tiling, see \cite[Section 4]{AP}.

In this paper we make use of the discrete hull of a tiling, a particular subset of the continuous hull. Let $T$ be a tiling with prototile set $\PP$. Following Kellendonk \cite[Section 2.1]{Kel1}, for each $p \in \PP$, choose a point in the interior of $\supp(p)$ called a \textbf{puncture} and denote it by $x(p)$. This naturally punctures tiles $t=p+y$ by $x(t) \coloneqq x(p)+y$ and defines sets of punctures for patches and tilings. The \textbf{discrete hull} of a tiling $T$ is given by
\[
\Omp \coloneqq \{T' \in \Omega_T \mid \text{there exists}\ t \in T'\ \text{with}\ x(t) = 0 \} \subset \Omega_T,
\]
i.e., the subset of tilings with a puncture over the origin of $\R^d$. If $T$ is repetitive, non-periodic and has FLC then $\Omp$ is a Cantor set. In particular, $\Omp$ is a compact metric space that has a basis of clopen sets. Indeed, for a patch $P$ and a tile $t$ in $P$, the set
\begin{equation}\label{disc_top}
U(P,t) \coloneqq \{ T' \in \Omp \mid P-x(t) \subset T' \}
\end{equation}
is clopen in $\Omp$, and the set of all such sets forms a basis for the metric topology on $\Omp$.

\subsection{The substitution graph and discrete hull}
For $p \in \PP$ and $t \in \sub(p)$ we call $(t,p)$ a {\bf supertile extension}, and denote the set of such supertile extensions by $\mc{S}$. If at most one copy of each prototile appears in each substituted prototile, then the elements of $\mc{S}$ can be identified with all pairs $(a,b) \in \mc{P}^2$ for which $a \in \sub(b)$, but we do not need to assume this in general (see Example \ref{ex: supertile duplicates} below). Given a supertile extension $e = (t,p) \in \mc{S}$, we denote $r(e) \coloneqq t$ (considered as a prototile in $\PP$) and $s(e) \coloneqq p$. We construct the {\bf substitution graph} $G$ with vertex set $\mc{P}$, edge set $\mc{S}$ and source and range maps $s$, $r \colon \mc{S} \to \mc{P}$. The associated set of right-infinite (left-pointing) paths is denoted
\[
\mc{F} \coloneqq \{e_0 e_1 e_2 \cdots \mid s(e_i) = r(e_{i+1})\},
\]
and comes equipped with the left shift map $\sigma: \mc{F} \to \mc{F}$ defined by $\sigma(e_0 e_1 e_2 \cdots)=e_1 e_2 \cdots$. Generally, given a finite graph $G$, the set $\mc{F}$ of right-infinite words as above is called a \emph{topological Markov chain}.

\begin{remark} \label{rem: composition order}
One could use the opposite convention to above, taking an arrow $e \colon t \to p$ as a `subtile inclusion' of a tile $t$ into a $p$ supertile. In that case, all arrows on graphs would be reversed and one could take $\mc{F}$ as right-infinite, right-pointing paths. The advantage of the convention we take here instead is that a string $e_0 e_1 \cdots e_n$ may be read analogously to function composition (with range on the left, source on the right), and when introducing semigroup elements, which also have domain/codomain or `in/out' tiles, a valid string has consistently matching adjacent tiles in the domain/codomain or source/range, whether the term is a semigroup element or supertile extension term. By this convention arrows point in the direction of substitution application, which is also similar to standard conventions on inverse limits defining the tiling space, as we shall see in Section \ref{sec: limit space}.
\end{remark}

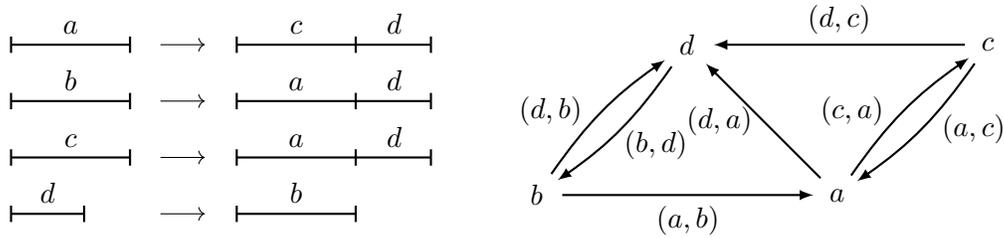
\begin{figure}
\begin{center}
\begin{tikzpicture}
\begin{scope}[xshift=0cm,yshift=2cm]
\draw[|-|,thick] (0,0)-- node[above]{$a$} (1.61,0);
\draw[|-|,thick] (3,0)-- node[above]{$c$} (4.61,0);
\draw[-|,thick] (4.61,0)-- node[above]{$d$} (5.61,0);
\draw[|-|,thick] (0,-0.75)-- node[above]{$b$} (1.61,-0.75);
\draw[|-|,thick] (3,-0.75)-- node[above]{$a$} (4.61,-0.75);
\draw[-|,thick] (4.61,-0.75)-- node[above]{$d$} (5.61,-0.75);
\draw[|-|,thick] (0,-1.5)-- node[above]{$c$} (1.61,-1.5);
\draw[|-|,thick] (3,-1.5)-- node[above]{$a$} (4.61,-1.5);
\draw[-|,thick] (4.61,-1.5)-- node[above]{$d$} (5.61,-1.5);
\draw[|-|,thick] (0,-2.25)-- node[above]{$d$} (1,-2.25);
\draw[|-|,thick] (3,-2.25)-- node[above]{$b$} (4.61,-2.25);
\draw[->] (2,-0)-- (2.59,0);
\draw[->] (2,-.75)-- (2.59,-.75);
\draw[->] (2,-1.5)-- (2.59,-1.5);
\draw[->] (2,-2.25)-- (2.59,-2.25);
\end{scope}
\begin{scope}[xshift=7cm,yshift=0cm]
\node[vertex] (vert_b) at (0,0)   {$b$};
\node[vertex] (vert_d) at (2,2)   {$d$}
	edge [<-,>=latex,out=215,in=55,thick] node[left,pos=0.5]{$(d,b) \ $} (vert_b)
	edge [->,>=latex,out=235,in=35,thick] node[right,pos=0.65]{$\ (b,d)$} (vert_b);
\node[vertex] (vert_c) at (6,2)   {$c$}
	edge [->,>=latex,out=180,in=0,thick] node[above,pos=0.5]{$(d,c)$} (vert_d);
\node[vertex] (vert_a) at (4,0)   {$a$}
	edge [<-,>=latex,out=35,in=235,thick] node[right,pos=0.5]{$\ (a,c)$} (vert_c)
	edge [->,>=latex,out=55,in=215,thick] node[left,pos=0.5]{$(c,a)\ $} (vert_c)
	edge [->,>=latex,out=135,in=-45,thick] node[left,pos=0.5]{$(d,a)$} (vert_d)
	edge [<-,>=latex,out=180,in=0,thick] node[below,pos=0.5]{$(a,b)$} (vert_b);
\end{scope}
\end{tikzpicture}
\end{center}
\caption{The border forcing Fibonacci substitution and its graph of supertile extensions.}
\label{fig:AP sub}
\end{figure}

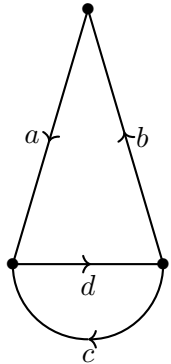
\begin{figure}
\begin{center}
\begin{tikzpicture}
\begin{scope}[xshift=0cm,yshift=0cm,decoration={markings,mark=at position 0.52 with {\arrow{>}}}]
\draw[-,thick,postaction={decorate}] (-1,0)-- node[below] {$d$} (1,0);
\draw[-,thick,postaction={decorate}] (1,0) -- node[right] {$b$} (0,3.387);
\draw[-,thick,postaction={decorate}] (0,3.387) -- node[left] {$a$} (-1,0);
\draw[-,thick] (-1,0) arc (180:270:1) node[below] {$c$};
\draw[<-,thick] (0,-1) arc (270:360:1);
\node at (-1,0) {$\bullet$};
\node at (1,0) {$\bullet$};
\node at (0,3.387) {$\bullet$};
\end{scope}
\end{tikzpicture}
\end{center}
\caption{The AP complex for the border forcing Fibonacci tiling \cite{AP}.}
\label{fig:AP complex fib}
\end{figure}

\begin{example}\label{ex: fib}
We illustrate supertile extensions and the substitution graph by studying a border forcing version of the Fibonacci tiling, as defined in \cite[Section 10.1]{AP}. In particular, starting with the usual Fibonacci substitution $0 \mapsto 01$ and $1 \mapsto 0$ we define our tiles as a sliding block code, with $a=0[0]1$, $b=1[0]0$, $c=1[0]1$ and $d=0[1]0$ (the terms in brackets are the tiles being collared, and terms to the left and right denote collaring information). From this we obtain the substitution $\varphi: \mc{P} \to \mc{P}^*$ given by
\begin{equation}\label{eq: Fib substitutions}
\varphi(a)=cd, \, \varphi(b)=ad, \, \varphi(c)=ad \text{ and } \varphi(d)=b.
\end{equation}
It is routine to check that this substitution is recognisable. Using \eqref{eq: Fib substitutions} we immediately obtain the supertile extensions:
\[
(c,a), (d,a), (a,b), (d,b), (a,c), (d,c) \text{ and } (b,d),
\]
which can unambiguously be labelled by pairs in $\mc{P}^2$ since at most one copy of each prototile appears in each supertile. The substitution appears in Figure \ref{fig:AP sub}, along with the graph associated with this substitution. Note that a smaller substitution (on a 3-letter alphabet) could be used to force the border for this example, by only collaring tiles on the left, since every supertile is adjacent to a $0$-tile on the right. \qed
\end{example}

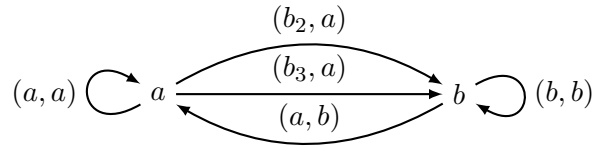
\begin{figure}
\begin{center}
\begin{tikzpicture}
\node (a) at (0,0)   {$a$};
\node (b) at (4,0)   {$b$};

\path[->,thick,>=latex]
(a) edge [out=30,in=150,"{$(b_2,a)$}"] (b)
(a) edge ["{$(b_3,a)$}"] (b)
(b) edge [out=210,in=-30] node[above] {$(a,b)$} (a)
(a) edge [out=210, in=150, looseness=10, "{$(a,a)$}"] (a)
(b) edge [out=30, in=-30, looseness=10, "{$(b,b)$}"] (b);
\end{tikzpicture}
\end{center}
\caption{The supertile extension graph of Example \ref{ex: supertile duplicates}.}
\label{fig: supertile duplicates}
\end{figure}

\begin{example} \label{ex: supertile duplicates}
Consider the substitution on $\mc{P} = \{a,b\}$ defined by
\begin{equation*}
\varphi(a)=abb \text{ and } \varphi(b)=ab.
\end{equation*}
It is easy to see that this substitution is recognisable and forces the border since every supertile is followed by an $a$ tile to its right and is preceded by a $b$ to its left. In this case there two ways tile $b$ is extended into an $a$ supertile: by the $b$ being included as either the second or third letter. These could be distinguished by denoting them as, say, $(b_2,a)$ and $(b_3,a)$. The substitution graph is given in Figure \ref{fig: supertile duplicates}. \qed
\end{example}

We have a natural bijection between a right infinite sequence of supertile extensions
\begin{equation}\label{string}
(t_0,t_1)(t_1,t_2)(t_2,t_3) \cdots \in \mc{F}
\end{equation}
and the discrete hull $\Omp$ of all substitution tilings $T$ generated by $\sub$ (where here, and later, we allow a very minor abuse of notation by denoting a supertile extension by $(t_n,t_{n+1})$ with each $t_n$ simultaneously denoting a subtile of $\sub(t_{n+1})$ and also its corresponding prototile in $\PP$). Indeed, to such a string \eqref{string} the puncture of tile $t_0$ is placed on the origin and we obtain a sequence of inclusions $t_0 \subseteq \sub(t_1) \subseteq \sub^2(t_2) \subseteq \sub^3(t_3) \subseteq \cdots$. The nested patches $\sub^n(t_n)$ determine the entire tiling by the border forcing property. Conversely, a punctured tiling determines such a string by recognisability (which itself follows from FLC and aperiodicity).

There is a natural topology on $\mc{F}$ whose basis consists of clopen cylinder sets of all infinite strings starting with some given finite initial string. Under this topology, the above bijection induces a homeomorphism
\begin{equation} \label{eq: strings <-> transversal}
\tau \colon \mc{F} \xrightarrow{\cong} \Omp
\end{equation}
to the discrete hull $\Omp$ with the topology generated by \eqref{disc_top}. The map $\tau$ above constructs an infinite tiling with a puncture on the origin from an infinite string. We may also define $\tau$ on a finite string $s = (t_0,t_1)(t_1,t_2)\cdots (t_{n-1},t_{n})$ to obtain a finite marked patch. This is the patch $\sub^n(t_n)$, where the location of $t_0$ is positioned in this supertile according to how the supertile extensions embed into each other.

The discrete hull $\Omp$ is the object space of the {\bf (discrete) translation groupoid} 
\[
\Rp \coloneqq \{ (T-x(t),T) \in \Omp \times \Omp \mid t \in T \},
\]
and the product topology coming from $\Omp \times \R^d$ makes $\Rp$ a principal topological groupoid. Given a patch $P$ and tiles $t,t'$ in $P$, the sets
\begin{equation}\label{Rpunc_top}
V(t',P,t) \coloneqq \{ (T',T) \in \Rp \mid P-x(t) \subset T \text{ and } P-x(t') \subset T' \}
\end{equation}
are open in $\Rp$, and the set of all such sets forms a basis for the product topology. In this topology, $\Rp$ is an \'etale equivalence relation. See \cite{Kel1} for further details.

The open set $V(t',P,t) \subset \Rp$ defines a bijection from $U(P,t)$ to $U(P,t')$, which may be thought of as the `partial translation' which shifts the origin tile from $t$ to $t'$, within any tiling of $\Omp$ with $t \in P$ centred over the origin. For example, moving across a certain face from one prototile to an adjacent one may be interpreted as such an open subset of the discrete groupoid, or as a partial bijection within $\Omp$ (or, via the identification $\tau$, within the Markov shift $\mc{F}$ for a substitution tiling). This collection of partial translations naturally leads us to the tiling semigroup:

\subsection{The tiling semigroup}

We recall Kellendonk's construction of the inverse semigroup associated to an FLC tiling \cite{KL,KelPut}. We use notation similar to \cite{KL} for the elements of this semigroup, the doubly pointed patches.

A semigroup $S$ is an \textbf{inverse semigroup} if for each $s \in S$ there exists a unique element $s^* \in S$ such that $ss^*s=s$ and $s^*ss^*=s^*$. According to the Vagner--Preston Representation Theorem \cite[Theorem V.1.10]{Howie}, every inverse semigroup is isomorphic to a subsemigroup of $\mathcal{I}(X)$, the inverse semigroup of partial bijections on a set $X$. An \emph{action} of an inverse semigroup $S$ on a set $X$ is a homomorphism $\pi:S \to \mathcal{I}(X)$. If the homomorphism $\pi$ is fixed, we usually write $g \cdot x$ for $\pi_g(x)$. See \cite{Exel} for further details.

\begin{definition}
A {\bf doubly pointed patch} $[b,P,a]$ is given by a finite patch $P$ and tiles $a,b \in P$, where we take the tuple $(b,P,a)$ up to translation equivalence. Let $\mc{T}$ be the set of all doubly pointed patches along with a `zero element' $0 \in \mc{T}$.

Let $[d,Q,c]$, $[b,P,a] \in \mc{T}$ be two doubly pointed patches which, without loss of generality (by translating each, if necessary) have $x(b) = x(c)$. If $P$ and $Q$ agree on any tiles with intersecting interiors and $P \cup Q$ is a valid patch, then we define
\[
[d,Q,c] [b,P,a] = [d,P \cup Q,a].
\]
Otherwise, we define $[d,Q,c] [b,P,a] = 0$. Any product with $0$ is defined as $0$. We call $\mc{T} = (\mc{T},\cdot)$ the {\bf tiling semigroup}.
\end{definition}

\begin{remark}
One could also define the tiling semigroup for a space $\Omega$ of tilings, considering patches over all tilings in $\Omega$. For $\Om = \Om_T$ with $T$ a repetitive tiling, all tilings of $\Om$ have the same finite patches, up to translation, so this does not affect the construction. However, if we consider $\Omega_\sub$ with $\sub$ non-primitive then the sets of finite patches can differ between orbits. We take the full collection of patches over all of $\Omega_\sub$ in such a case. We note, typically we are most interested in the case where $\sub$ is primitive and $\Omega_\sub = \Omega_T$ for any $T \in \Omega_\sub$, with each such being repetitive.
\end{remark}

Notice that, again, our notation here mirrors function composition: the element $[b,P,a]$ has `in tile' $a$ and `out tile' $b$, with a product of elements $[d,Q,c][b,P,a]$ interpreted as applying the right then the left-hand term, and requiring that the intermediate tiles $b$ and $c$ agree. There is a bijective correspondence between elements of $\mc{T}$ and basis elements of the \'etale topology for $\Rp$ as described in \eqref{Rpunc_top} via $[b,P,a] \mapsto V(b,P,a)$ (and where $0$ has empty (co)domain). The product of semigroup elements is identified with the composition of partial bijections on the largest compatible domain. Thus, the tiling semigroup $\mc{T}$ naturally acts by partial bijections on the discrete hull $\Omp$, where a doubly pointed patch $g = [b,P,a]$ has domain $U(P,a) \subset \Omp$ and codomain $U(P,b) \subset \Omp$.

It is easy to establish that $\mc{T}$ is an inverse semigroup. In particular, note that for $s = [b,P,a] \in \mc{T}$, the unique $t \in \mc{T}$ with $s = sts$ and $t = tst$ is given by $t = [a,P,b]$. We interpret $[b,P,a]$ as a \emph{translation from $a$ to $b$, within $P$}, where the product of two such translations is allowed when the union is itself a valid patch. The idempotents (or `partial identities') are of the form $[a,P,a]$, along with the $0$ element. Since $\mc{T}$ is an inverse semigroup, it naturally inherits a partial ordering: we have that $[b,P,a] \preceq [d,Q,c]$ if and only if, up to translation, we have an inclusion $(d,Q,c) \subseteq (b,P,a)$ of doubly pointed patches; that is, $P$ extends $Q$ as a doubly pointed patch with $a=c$ and $b=d$. It is not hard to see that $\mc{T}$ is generated by idempotents $[p,\{p\},p]$ for $p \in \mc{P}$, where $\{p\}$ is a single-tile patch, and elements $[b,P,a]$, where $P$ is a connected two-tile patch containing distinct $a$ and $b$. 

\section{Self-similarity of substitution tiling semigroups} \label{sec: Self-similarity of substitution tiling semigroups}

We now show that the above semigroup action of $\mc{T}$ on $\mc{F} \cong \Omp$ is self-similar.

\begin{definition}[{\cite[Definition 3.6]{BGN}}] \label{def: self-similar}
Let $\mc{F}$ be a topological Markov chain over an alphabet $X$. An inverse semigroup $G$ acting on $\mc{F}$ is called {\bf self-similar} if for every $g \in G$ and $x \in X$ there exist $y_1$, \ldots, $y_k \in X$ and $h_1$, \ldots, $h_k \in G$ such that the sets $\dom(h_i)$ are disjoint, $\bigcup_{i=1}^k x \dom(h_i) = x\mc{F} \cap \dom(g)$, and for every $xw \in \mc{F}$ we have
\begin{equation} \label{eq: self-sim}
g \cdot xw = y_i (h_i \cdot w),
\end{equation}
where $i$ is such that $w \in \dom(h_i)$.
\end{definition}

Often in the context of self-similar groups and semigroups the action of $g$ on $w$ is denoted by $w^g$, but here we choose to use $g \cdot w$. Note that the $h_i$ in \eqref{eq: self-sim} are not uniquely defined. Indeed, given a partial bijection $h_i$ one could partition its domain and use instead the restrictions of $h_i$ to each such subset. From the opposite perspective, one may always replace the expression $g \cdot w$ with $h \cdot w$ whenever $h$ is an extension of $g$ to a larger domain. This will be useful later in simplifying the semigroup elements $h_i$ generated by successive application of \eqref{eq: self-sim}.

\begin{remark}
We briefly explain here an equivalent description of self-similarity to highlight its algorithmic quality (and note that one may equivalently define self-similarity of inverse semigroup actions by automata, see \cite{Nek_SSIS}). We make the standing assumption throughout that all semigroup elements have clopen domains. For each $g \in G$, there is some $N(g) = N \in \N$, given by the distance required to `read forwards' in the sequence to evaluate the first letter of $g \cdot w$, for an infinite word $w \in \mc{F}$, as well as determining the necessary semigroup element to apply to the remainder of the string. Let $\mc{F}^N$ denote the set of words of length $N$. Self-similarity means that there exist letters $y_1$, \ldots, $y_\ell \in X$, elements $h_1$, \ldots, $h_\ell \in G$ and subsets $S_i \subseteq \mc{F}^N$ so that:
\begin{enumerate}
	\item $\dom(g) = \bigcup_{i=1}^\ell S_i \mc{F}$, where $S_i \mc{F}$ denotes the set of infinite words with initial $N$-letter string in $S_i$;
	\item $S_i \cap S_j = \emptyset$ for $i \neq j$ (so the above is a disjoint union);
	\item for each $i = 1$, \ldots, $\ell$ we have $S_i \mc{F} = x_i\dom(h_i)$ for some $x_i \in X$;
\end{enumerate}
Then given an infinite word $w \in \mc{F}$, to evaluate $g \cdot w$ we first determine the initial $N$-letter string $s$ of $w$. We have that $s \in S_i$ for a unique $i$, and then we have the rule:
\[
g \cdot w = y_i (h_i \cdot \sigma(w)),
\]
where $\sigma \colon \mc{F} \to \mc{F}$ is the left shift (i.e., the map removing the initial letter $x_i$ from $w$). Thus, the first letter of $g \cdot w$ is $y_i$. To apply $h_i$ to the remainder, we look forward distance $N(h_i)$ in $\sigma(w)$ to determine the second letter of $g \cdot w$, as well as the next element of $G$ to apply to $\sigma^2(w)$. This may be repeated indefinitely. This is best demonstrated here through Example \ref{geometric intuition}, which may help the reader with the following proof.
\end{remark}

\begin{theorem}
The tiling semigroup $\mc{T}$ of a recognisable substitution tiling is self-similar.
\end{theorem}

\begin{proof}
Let $g = [b,P,a]$ be a doubly pointed patch and recall that the domain of $g$ is $U(P,a)$ corresponding to tilings with patch $P$ at the origin centred at the puncture $x(a)$ of tile $a$. Since $\sub$ forces the border there is some $N \in \N$ so that any $w \in \mc{F}$, with $\tau(w) \in U(P,a)$, has initial string $s \in \mc{F}^N$ with $P \subseteq \tau(s)$. In particular, there is a finite set $S = \{s_1,\ldots,s_k\} \subseteq \mc{F}^N$ of all possible length-$N$ strings for which $P \subseteq \tau(s_i)$.

Fix $s_i \in \mc{F}^N$ and define $h_i \in \mc{T}$ by
\[
h_i = [b',\tau(s_i'),a'],
\]
where $s'_i \in \mc{F}^{N-1}$ is given by removing the first term of $s_i$, and $a'$ and $b'$ are the tiles of $\tau(s_i')$ whose substitutes contain the tiles $a$, $b \in P$. Let $y_i \in \mc{S}$ denote the unique supertile extension that includes $b \in P \subset \tau(s_i)$ into its 1-supertile $b' \in \tau(s_i')$.

We may now check that the above assignments fulfil the definition of self-similarity. Take any $w \in \dom(g)$. Then $w$ has initial $N$-letter string $s_i \in S$. We must check that
\begin{equation}\label{tiling_SS}
g \cdot w = y_i( h_i \cdot \sigma(w)).
\end{equation}
By definition, $\tau(w)$ is a tiling $T \in U(P,a)$ and $\tau(g\cdot w)=T-x(b)$ so that $T-x(b) \in U(P,b)$. Let $T' = \tau(\sigma(w))$ which is the 1-supertiling of $T$ with the puncture of tile $a'$ at the origin. Then our definition of $h_i$ implies that $\tau(h_i \cdot \sigma(w))$ is the 1-supertiling with the puncture of tile $b'$ at the origin. Pre-appending $y_i$ corresponds to substituting this tiling, translated appropriately to position the puncture of $b$ over the origin. Thus \eqref{tiling_SS} holds, as required.
\end{proof}

\begin{example}\label{geometric intuition}
We consider the Fibonacci substitution from Example \ref{ex: fib}. We denote by $P_{xy}$ the two-tile patch consisting of $x$, $y \in \mc{P}$, with $x$ on the left and $y$ on the right. Later, in Example \ref{ex: fib rules}, we give a complete set of rules on applying doubly-pointed two-tile patches to strings. To be applied, such elements need to look at either the next term, or the next two terms. However, we quickly give an example application to give the flavour of the definitions above, where one sees the group element working algorithmically through the string:

\begin{figure}
\begin{center}
\begin{tikzpicture}
\begin{scope}[xshift=0cm,yshift=0cm]
\draw[|-|,dashed] (-5.22,0) -- node[below,pos=0.5]{$a$} (-3.61,0);
\draw[|-|,dashed] (-3.61,0) -- node[below,pos=0.5]{$d$} (-2.61,0);
\draw[|-|] (-2.61,0) -- node[below,pos=0.5]{$c$} (-1,0);
\draw[|-|] (-1,0) -- node[below,pos=0.5]{$d$} (0,0);
\draw[|-|] (0,0) -- node[below,pos=0.5]{$b$} (1.61,0);
\draw[|-|,dashed] (1.61,0) -- node[below,pos=0.5]{$a$} (3.22,0);
\draw[|-|,dashed] (3.22,0) -- node[below,pos=0.5]{$d$} (4.22,0);
\draw[|-|,dashed] (4.22,0) -- node[below,pos=0.5]{$b$} (5.83,0);
\draw[|-|,dashed] (5.83,0) -- node[below,pos=0.5]{$a$} (7.44,0);
\draw[|-|,dashed] (7.44,0) -- node[below,pos=0.5]{$d$} (8.44,0);
\draw[|-|,dashed] (-5.22,2) -- node[below,pos=0.5]{$b$} (-2.61,2);
\draw[|-|] (-2.61,2) -- node[below,pos=0.5]{$a$} (0,2);
\draw[|-|] (0,2) -- node[below,pos=0.6]{$d$} (1.61,2);
\draw[|-|,dashed] (1.61,2) -- node[below,pos=0.6]{$c$} (4.22,2);
\draw[|-|,dashed] (4.22,2) -- node[below,pos=0.5]{$d$} (5.83,2);
\draw[|-|,dashed] (5.83,2) -- node[below,pos=0.5]{$b$} (8.44,2);
\draw[|-|,dashed] (-5.22,4) -- node[below,pos=0.5]{$d$} (-2.61,4);
\draw[|-|] (-2.61,4) -- node[below,pos=0.45]{$b$} (1.61,4);
\draw[|-|,dashed] (1.61,4) -- node[below,pos=0.6]{$a$} (5.83,4);
\draw[|-|,dashed] (5.83,4) -- node[below,pos=0.6]{$d$} (8.44,4);
\draw[|-|] (-2.61,6) -- node[below,pos=0.45]{$d$} (1.61,6);
\node[vertex] (vert_l) at (0.8,0) {};
\node[vertex] (vert_r) at (2.41,0) {}
	edge [<-,>=latex,out=135,in=45,thick] node[above,pos=0.5]{$[a,P_{ba},b]$} (vert_l);
\node[vertex] (vert_ul) at (0.8,2)  {}
	edge [->,>=latex,out=270,in=90,thick] node[left,pos=0.5]{$(b,d)$} (vert_l);
\node[vertex] (vert_ur) at (3,2)  {}
	edge [<-,>=latex,out=135,in=45,thick] node[above,pos=0.5]{$[c,P_{dc},d]$} (vert_ul)
	edge [->,>=latex,out=270,in=90,thick] node[right,pos=0.5]{$(a,c)$} (vert_r);
\node[vertex] (vert_uul) at (-0.4,4)  {}
	edge [->,>=latex,out=270,in=90,thick] node[left,pos=0.5]{$(d,b)$} (vert_ul);
\node[vertex] (vert_uur) at (3.8,4)  {}
	edge [<-,>=latex,out=135,in=45,thick] node[above,pos=0.5]{$[a,P_{ba},b]$} (vert_uul)
	edge [->,>=latex,out=270,in=90,thick] node[right,pos=0.5]{$(c,a)$} (vert_ur);
\node[vertex] (vert_uuul) at (-0.4,6)  {}
	edge [->,>=latex,out=270,in=90,thick] node[left,pos=0.5]{$(b,d)$} (vert_uul);
\node at (0,0) {$\ddag$};
\node at (1.61,0) {$\ddag$};
\node at (4.22,0) {$\ddag$};
\node at (5.83,0) {$\ddag$};
\node at (8.44,0) {$\ddag$};
\node at (-5.22,0) {$\ddag$};
\node at (-2.61,0) {$\ddag$};
\node at (1.61,2) {$\ddag$};
\node at (-2.61,2) {$\ddag$};
\node at (-5.22,2) {$\ddag$};
\node at (5.83,2) {$\ddag$};
\node at (1.61,4) {$\ddag$};
\node at (8.44,4) {$\ddag$};
\node at (-2.61,4) {$\ddag$};
\node at (1.61,6) {$\ddag$};
\end{scope}
\end{tikzpicture}
\end{center}
\caption{A graphical representation of the patch formed by the prefix $(b,d)(d,b)(b,d)$ of the word $w$. Tile lengths are increased by the golden ratio at each increasing level. The vertical arrows on the left denote the tile inclusions $(b,d)(d,b)(b,d)$. The bottom horizontal arrow depicts the action of moving one tile to the right by a semigroup element and a non-idempotent element of $\mc{T}$ acts on the remainder of the string if we translate across a supertile boundary, denoted by double daggers. Thus, the vertical arrows on the right represent the output of applying the semigroup element $[a,P_{ba},b]$ to $(b,d)(d,b)(b,d)$ as shown in \eqref{fib_comp_2}. Note that the solid tiles are given directly from the tile inclusions and the dashed tiles are defined implicitly by the border forcing property.}
\label{fig:Fib_tiling_example}
\end{figure}
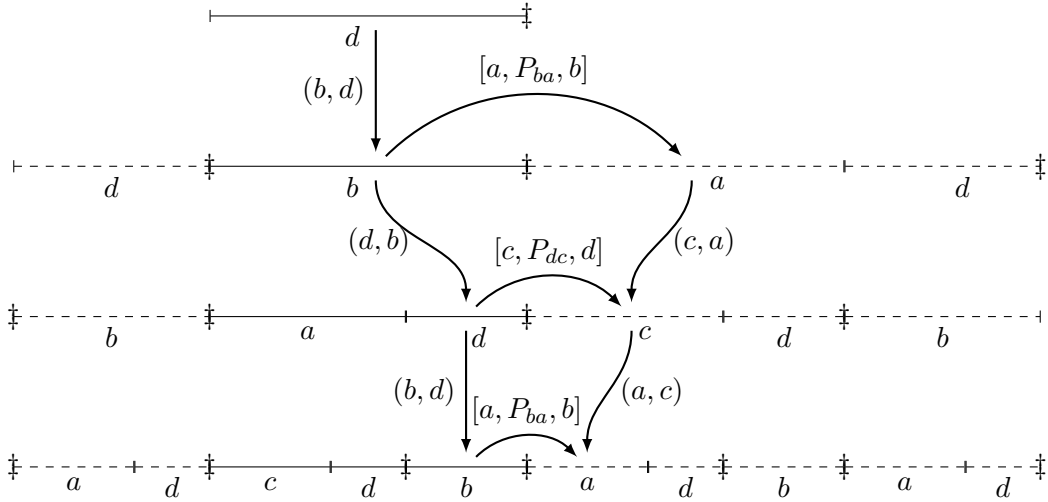

Take the infinite word
\[
w = (b,d)(d,b)(b,d) e f z \in \mc{F}
\]
where $(b,d)$, $(d,b)$ and $(b,d)$ are specified supertile extensions in $\mc{S}$, $e$ and $f$ are in $\mc{S}$ and $z \in \mc{F}$ is an infinite tail. We apply the element $[a,P_{ba},b]$ to $w$ according to the rules found in Example \ref{ex: fib}, which corresponds to moving one tile to the right in the tiling formed by $w$. It is important to note that all these relations can be deduced from the information given, as depicted in Figure \ref{fig:Fib_tiling_example}.
\begin{align}
\notag
[a,P_{ba},b] \cdot (b,d)(d,b)(b,d)e f z  & =(a,c) \left( [c,P_{dc},d] \cdot (d,b) (b,d) efz \right) \\
\label{fib_comp_2}
&=(a,c)(c,a) \left( [a,P_{ba},b] \cdot (b,d)e f z \right)
\end{align}

At this stage, we cannot evaluate $[a,P_{ba},b]$ without knowing $e$. If, for example, $e=(d,a)$ and $f=(a,b)$ then one finds that \eqref{fib_comp_2} evaluates as
\begin{align*}
(a,c)(c,a) \left( [a,P_{ba},b] \cdot (b,d)(d,a)(a,b) z \right) &=(a,c)(c,a)(a,b) \left( [b,P_{db},d] \cdot (d,a)(a,b) z \right) \\ 
&=(a,c)(c,a)(a,b)(b,d) \left( [d,P_{ad},a] \cdot (a,b) z \right) \\ 
&=(a,c)(c,a)(a,b)(b,d)(d,b) \left( [b,P_{b},b] \cdot  z \right) \\ 
&=(a,c)(c,a)(a,b)(b,d)(d,b)z \\ 
\end{align*}
where the last line is fully evaluated, since the idempotent $[b,P_b,b]$ acts as the identity. If instead we took $e=(d,b)$ then one finds that \eqref{fib_comp_2} evaluates as
\begin{align*}
(a,c)(c,a) \left( [a,P_{ba},b] \cdot (b,d)(d,b)f z \right) &=(a,c)(c,a)(a,c) \left( [c,P_{dc},d] \cdot (d,b)f z \right) \\
&=(a,c)(c,a)(a,c)(c,a) \left( [a,P_{ba},b] \cdot f z \right) 
\end{align*}
and we cannot evaluate further without knowing both $f$ and the first letter of $z$.

If we translate left one tile instead, by applying $[d,P_{db},b]$ to $w$, then we find:
\begin{align*}
[d,P_{db},b] \cdot (b,d) (d,b)(b,d) e f z &=(d,a) \left( [a,P_{ad},d] \cdot (d,b) (b,d)e f z\right)\\
&= (d,a) (a,b) \left( [b,P_b,b] \cdot (b,d)  e f z \right) \\
&= (d,a) (a,b)(b,d) e f z
\end{align*}
and the application is fully evaluated, again because $[b,P_b,b]$ is an idempotent. We encourage the reader to consider translating to the left in Figure \ref{fig:Fib_tiling_example} to work out similar equations.

Notice the necessary consistency in the above strings: there is agreement between adjacent tile types of both the supertile extension pairs $(y,x)$ and the `in/out' tiles of the doubly pointed patches $[y,P,x]$. As in Remark \ref{rem: composition order}, this follows from the convention of orientation in the substitution graph and writing strings in an order corresponding to function composition.
\qed
\end{example}

One should observe that, in terms of the action of $\mc{T}$ on $\mc{F}$, there is some degree of superfluous information in the semigroup elements:  for $g \in \mc{T}$ and $T \in \dom(g)$, all that is required to evaluate $g(T)$ is the relative displacement of the `in and out' tiles. This fact is also reflected algebraically in terms of the inverse semi-group: For a general inverse semigroup one has the partial ordering defined by letting $x \preceq y$ if there is an idempotent $e$ for which $x = ey$. For the tiling semigroup $\mc{T}$, this says that $[b,P,a] \preceq [d,Q,c]$ if and only if, up to translation, we have an inclusion of doubly pointed patches $(d,Q,c) \subseteq (b,P,a)$ (i.e., $P$ extends $Q$ with $a=c$ and $b=d$). If $x \preceq y$ then $\dom(x) \subseteq \dom(y)$ and $x \cdot \tau^{-1}(T) = y \cdot \tau^{-1}(T)$ for all $T \in U(P,a)$.

A consequence of the above is that there is significant choice of semigroup elements satisfying the self-similarity rule \eqref{eq: self-sim}. This is easily dealt with in practice since, if we ignore the domain and range of elements, we may always replace a term such as $h_i \cdot w$ with $j \cdot w$ for any $j \succeq h_i$ in \eqref{eq: self-sim}. In fact, after enough iterations, we see that we may take $j$ to be a `small' patch. This is made precise via the following more general definition.

\begin{definition} \label{def: contracting}
Let $(G,\mc{F})$ be a self-similar inverse semigroup. We call $G$ {\bf contracting} if there exists some finite $N \subseteq G$ satisfying the following: For any $g \in G$ there exists some $k \in \N$ for which, for any $uw \in dom(g)$ with $u \in \mc{F}^k$, there exists some $v \in \mc{F}^k$ and $h \in N$ with
\begin{equation} \label{eq: contracting}
g \cdot (uw) = v (h \cdot w).
\end{equation}
\end{definition}

\begin{remark}
\label{rem: alt contracting}
We can write the above definition in the following alternative way. There exists some finite $N \subseteq G$ satisfying the following: For any $g \in G$ there exists some $k \in \N$ for which, for every $w \in \dom(g)$, there is some $v \in \mc{F}^k$ and $h \in N$ with
\begin{equation} \label{eq: alt contracting}
g \cdot w = v (h \cdot (\sigma^k w)).
\end{equation}
\end{remark}

In the standard language of self-similar group actions, the above says that after sufficiently many applications of the `restriction' of $g$ (the elements $h_i$ of \eqref{eq: self-sim}) the new semigroup element to apply to the remainder of the string may be taken in the finite set $N$, at least after an appropriate adjustment of its domain and range.

\begin{definition}
Let $(G,\mc{F})$ be a contracting self-similar inverse semigroup. Call $\NN$ a \emph{semi-nucleus} if it satisfies the contracting condition above and is such that for all $g \in \NN$ and $ew \in \dom(g)$ with $e \in \mc{S} = \mc{F}^1$, there is some $f \in \mc{S}$ and $h \in \NN$ so that
\begin{equation} \label{eq: semi-nucleus}
g \cdot (ew) = f (h \cdot w).
\end{equation}
\end{definition}

The above says that not only do all semigroup elements eventually `restrict' to elements in $\NN$ (possibly after extended the domain and range), we also have that a single iteration of restriction of an element of $\NN$ can be chosen to remain in $\NN$. That is, we may take $k=1$ for any $g \in N$ in Definition \ref{def: contracting}. 

\begin{lemma} \label{lem: semi-nucleus}
A contractive self-similar inverse semigroup has a semi-nucleus.
\end{lemma}

\begin{proof}
Let $N \subseteq G$ be as required for the definition of contractivity and $g \in N$, satisfying \eqref{eq: alt contracting} for $k  = k(g) \in \N$. Let $u = u_1 \cdots u_k \in \mc{F}^k$ and $w \in \mc{F}$ with $uw \in \dom(g)$.

Consider the elements $h_1^1$, $h_2^1$, \ldots, $h_n^1 \in G$ arising from \eqref{eq: self-sim} for $g$ and $x = u_1$. For each $h_i^1$, apply \eqref{eq: self-sim} again with $x = u_2$ to obtain elements $h_1^2$, $h_2^2$, \ldots, $h_m^2$. We continue this procedure to generate elements $h_i^j$ for $j = 1$, \ldots, $k$.

Iteratively applying self-similarity we have
\[
g \cdot (u_1 u_2 \cdots u_k w) = \cdots = y_1 \cdots y_{k-1} (h_\ell^{k-1} \cdot (u_k \cdot w)) =  y_1 \cdots y_k (h_j^k \cdot w) = v(h \cdot w),
\]
for $v = y_1 \cdots y_k \in \mc{F}^k$ and $h \in N$, by \eqref{eq: contracting}. This shows, at least in the above expression, that we may replace $h_j^k$ with $h \in N$. In fact, again by repeated application of self-similarity and with more careful consideration of the domains, we have
\[
\dom g \cap u \mc{F} = \bigsqcup_\ell u (\dom h_\ell^k),
\]
where the above is a disjoint union and $g \cdot (uw) = v (h_\ell^k \cdot w)$ for all $w \in \dom h_\ell^k$. Thus, for all $w \in \dom h_\ell^k$, we have $h_\ell^k \cdot w = h \cdot w$ for some $h \in N$.

Thus, let $\NN$ be the union of $N$ and all elements $h_i^j$, for $j < k(g)$, generated by $k(g)-1$ applications of the self-similar rule for each $g \in N$. Then \eqref{eq: semi-nucleus} holds for $g \in N$ using some $h = h_i^1 \in \NN$, by construction. Similarly, for $j < k(g)-1$, each $h_i^j \in \NN$ satisfies \eqref{eq: semi-nucleus} with $h = h_\ell^{j+1} \in \NN$ for some $\ell$. Finally, for $j = k(g)-1$, \eqref{eq: semi-nucleus} is satisfied for each $h_\ell^{k(g)-1} \in \NN$ using some $h \in N \subseteq \NN$, as above, so that $\NN$ is a semi-nucleus.
\end{proof}

\begin{remark} \label{rem: disconnected patches}
For the tiling semigroup, we could define a larger inverse semigroup which does not demand that patches are connected. Then for every (non-zero) $g \in \mc{T}$ we have $g \preceq z$ for some $z = [b,P,a]$ with $P$ a two-tile patch containing $a$ and $b$ (or a $1$-tile patch, if $a=b$). We will sometimes make temporary use of such partial translations not in $\mc{T}$, such as in the proof below.
\end{remark}

\begin{proposition} \label{prop: tiling semigroup contractive}
Let $(\mc{T},\mc{F})$ be the self-similar tiling semigroup of an aperiodic substitution tiling. Then for each $g \in \mc{T}$ there is some $k = k(g) \in \N$ so that, for all $uw \in \dom(g)$ with $u \in \mc{F}^k$, we have that
\[
g \cdot (uw) = v (h \cdot w)
\]
for some $v \in \mc{F}^k$ and $h = [b,P,a]$ for which $a=b$ or $a$ and $b$ are adjacent. Hence, $\mc{T}$ defines a contractive action on $\mc{F}$. We may choose $\NN$ to be a semi-nucleus consisting of all doubly-pointed {\bf star patches} $[b,P,a]$, that is, with $P$ a patch of tiles all sharing a common point.
\end{proposition}

\begin{proof}
Without loss of generality, we can take $g=[b,P,a]$ where $P$ is a two-tile patch containing only tiles $a$ and $b$ (but with $P$ possibly disconnected, see Remark \ref{rem: disconnected patches}). Let $g=h_0$ and consider any element $w \in \dom(g)$ along with a sequence $\{h_i\} \subset \mc{T}$ arising from recursively applying the self-similar relation \eqref{eq: self-sim} to $g \cdot w$. We may choose each $h_i=[b_i,P_i,a_i]$ using a one or (possibly disconnected) two-tile patch.

Consider the sequence of tile pairs $\{(a_i,b_i)\}$ coming from the doubly pointed patches $h_i=[b_i,P_i,a_i]$. Let $r_i \coloneqq \inf\{|x-y| \mid x \in \supp(a_i) \text{ and } y \in \supp(b_i)\}$ be the distance between tiles $a_i$ and $b_i$; that is, the infimum of distances between points of $a_i$ and $b_i$. It follows from the proof of Theorem \ref{tiling_SS} that the supertile extensions (of $a_{i-1}$ into $a_i$ and $b_{i-1}$ into $b_i$) arising from the self-similar relation \eqref{eq: self-sim} geometrically embed the two-tile patch $(a_{i-1},b_{i-1})$ as a subpatch of the substitution of $(a_i,b_i)$. It follows immediately that $r_i \leq \lambda^{-1} r_{i-1}$ for all $i \in \N$ since pairs of supertile extensions within a patch uniformly scale tiles by $\lambda^{-1}$ between the range and source, and the support of the source covers the support of the range.

We claim the sequence $(r_i)$ is eventually zero. If not, then this would provide an infinite sequence of two-tile patches with arbitrarily small distance apart. But it follows from FLC that $r_i$ can only take finitely many values less than $r_0$, so this cannot happen. Hence, we have that $r_i = 0$ for $i > k$ where, by FLC, $k$ can be taken to only depend on $(a,b)$ and thus on $g$. Since $r_i = 0$ if and only if $a_i$ and $b_i$ are equal or adjacent, it follows that $h_i$ may be given by a doubly pointed one or two-tile patch of intersecting tiles for $i > k$.

Let $\NN$ be the set of doubly pointed star-patches. By the above, for each $g \in \mc{T}$ there is some $k \in \N$ satisfying \eqref{eq: contracting} with $h \in \NN$. Since the source and range tiles of the $h_i$ intersect also for $i > k$, further restrictions may be taken in $\NN$, which is thus a semi-nucleus for $\mc{T}$.
\end{proof}

\begin{remark}
For self-similar groups one defines \emph{the} nucleus of a contracting group $G$ to be the minimal $\NN$ so that all elements eventually restrict to $\NN$. In \cite{Nek_SSIS}, a notion of contractivity and (minimal, uniquely defined) nucleus is given in the case of self-similar semigroups. However, in this setup self-similar semigroups are treated via automatons which are $\omega$-deterministic, which amounts to declaring fixed restrictions of partial bijections in \eqref{eq: self-sim} as part of the structure. In our setup we have allowed this to remain flexible, which is an alternative approach which we feel may be of further interest. Indeed, it was beneficial in the proof above that it was not necessary to manage the shapes of patches under restriction down to the semi-nucleus. It is also clear that, in this setting, it may be impossible and unnatural to have a unique and minimal semi-nucleus $\NN$. For example, for a cellular $2$-dimensional tiling, if we define connected patches via meeting tiles merely being adjacent, then we only need $1$ and $2$-tile doubly pointed patches in $\NN$. If we instead define tiles to be meeting when they meet over a shared edge, then star patches $[b,P,a]$ with $a$ and $b$ meeting at a shared vertex (but not over an edge) can be removed, and replaced with star patches $[b,P',a]$, with $P' \subset P$ connected, for which there is some degree of arbitrary choice.
\end{remark}

In the case of a $d$-dimensional cellular tiling, it is not hard to see that $\mc{T}$ is generated by idempotents (which may be identified with $\mc{P}$) together with $\mc{P}_2$, defined as the finite set of elements $[b,P,a]$ for $a \neq b$ and $P$ a two-tile patch consisting of $a$ and $b$ meeting over a particular shared $(d-1)$-dimensional face. Idempotents restrict to idempotents, and elements of $\mc{P}_2$ restrict to elements of $\mc{P} \cup \mc{P}_2$ (after possibly extending domains). So the action of $\mc{T}$ on $\mc{F}$ may be completely described by the action of the finite set $\mc{P}_2$ on strings of sufficiently large length, together with how they restrict to elements of $\mc{P} \cup \mc{P}_2$. However, the semi-nucleus still requires more elements for tilings of dimension greater than one, since the restriction of a `diagonally adjacent doubly pointed patch' can remain as such after arbitrarily many restrictions.

\section{The limit space}\label{sec: limit space}

Let $G$ be a finite graph with associated topological Markov shift $\mc{F}$ (the right-infinite, left-pointing paths). We define
\[
\mc{F}^- \coloneqq \{\cdots e_{-3}e_{-2}e_{-1} \mid r(e_i) = s(e_{i-1})\};
\]
that is, the space of left-infinite, left-pointing paths, which is equipped with the product topology. The following is a natural adaptation of the asymptotic equivalence relation from the case of self-similar groups to semigroups:

\begin{definition}\label{def:ae}
Two elements $x = \cdots e_{-3}e_{-2}e_{-1}$ and $y = \cdots f_{-3}f_{-2}f_{-1} \in \mc{F}^-$ are called {\bf asymptotically equivalent} with respect to the action of the semigroup $G$ if there is a sequence $(g_n)$ of $G$, with $\{g_n\} \subseteq G$ finite, and some $w \in \mc{F}$ so that for each $n \in \N$ the element
\begin{equation} \label{eq: ae relation}
g_n \cdot (e_{-n} \cdots e_{-3}e_{-2}e_{-1}w) \in \mc{F}
\end{equation}
has initial string of $n$ terms given by $f_{-n} \cdots f_{-3}f_{-2}f_{-1} \in \mc{F}^n$. In this case we write $x \aeq y$. We define the {\bf asymptotic equivalence relation} $\sim$ on $\mc{F}^{-}$ to be the equivalence relation generated by $\aeq$.
\end{definition}

The main difference between the above definition and the case of self-similar groups is that we need to append the infinite word $w$ to the right of the finite string $e_{-n} \cdots e_{-1}$ so that $g_n$ may be unambiguously applied to it. However, by self-similarity, it is in fact only necessary to append a finite string of sufficiently large length.

In the lemma below, and henceforth, we will always assume that for each $x \in \mc{F}$ there is some $g_x \in G$ with $x \in \dom(g_x)$ and $\dom(g_x)$ open.

\begin{lemma} \label{lem: aeq ref sym}
Let $G$ be a self-similar inverse semigroup acting on the Markov chain $\mc{F}$. Then $\aeq$ is reflexive. Suppose that $G$ is contractive and $e \aeq f$. Then we may take each $g_n \in \NN$ in \eqref{eq: ae relation} for $\NN$ some semi-nucleus, and there exists $w \in \mc{F}$ and $h \in \NN$, not depending on $n$, such that
\begin{equation} \label{eq: aeq alt}
g_n \cdot (e_{-n} \cdots e_{-3}e_{-2}e_{-1}w) = f_{-n} \cdots f_{-3}f_{-2}f_{-1} (h \cdot w)
\end{equation}
In particular, $\aeq$ is symmetric.
\end{lemma}

\begin{proof}
By compactness, one can choose a finite number of $x \in \mc{F}$ with the union of $\dom(g_x)$ covering $\mc{F}$. We have idempotents $h_i = g_x^{-1} g_x$, $i=1$, \ldots $k$, which still have domains $\dom(h_i) = \dom(g_x)$ covering $\mc{F}$. Given $e = \cdots e_{-2} e_{-1} \in \mc{F}^-$ take any $w \in \ran({e_{-1}})$. Then we may take each $g_n$ to be some $h_i$, with $e_{-n} \cdots e_{-1} w \in \dom(h_i)$. Since each $h_i$ is an idempotent, we have that $g_n (e_{-n} \cdots e_{-1} w) = e_{-n} \cdots e_{-1} w$, so $e \aeq e$.

Suppose now that $G$ is contractive. By Lemma \ref{lem: semi-nucleus} we may choose a semi-nucleus $\NN$ for $G$. Given $g_n$, we have some $k(g_n) \in \N$ as required from Definition \ref{def: contracting}. By finiteness of $\{g_n\}$, we may take $K = \max_n k(g_n) < \infty$. Then
\[
g_{n+K} \cdot (e_{-(n+K)} \cdots \cdots e_{-1} w) = f_{-(n+K)} \cdots f_{-(n+1)} h \cdot (e_{-n} \cdots e_{-1} w) = f_{-(n+K)} \cdots f_{-1} w',
\]
for some $w' \in \mc{F}$ and $h \in \NN$, so we may suppose without loss of generality that $g_n = h \in \NN$. Since this applies for each $n \in \N$, and $\NN$ is finite, we see that we may take $\{g_n\}$ as a sequence in $\NN$.

By repeated application of \eqref{eq: semi-nucleus}, for each $n \in \N$ we may write
\begin{equation} \label{eq: aeq sym}
g_n \cdot (e_{-n} \cdots e_{-3}e_{-2}e_{-1}w) = f_{-n} \cdots f_{-3}f_{-2} f_{-1} (h \cdot w),
\end{equation}
where $h \in \NN$. By finiteness of $\NN$, some $h \in \NN$ as above occurs for infinitely many $n$. For each $n$ in this subsequence, we similarly have
\[
g_n \cdot (e_{-n} \cdots e_{-3}e_{-2}e_{-1}w) = f_{-n} \cdots f_{-3}f_{-2} h_1 \cdot (e_{-1} w)
\]
for some $h_1 \in \NN$. Some such $h_1$ occurs infinitely often, and we may take $g_1 = h_1$. Repeating for each $n \in \N$, the resulting Cantor diagonalisation argument implies that we may take each $g_n$ so that application of \eqref{eq: semi-nucleus} restricts each $g_n$ to $g_{n-1}$, with the final restriction to the right-infinite tail $h \cdot w$ not depending on $n$, as required.

Finally, applying $g_n^{-1}$ to both sides of \eqref{eq: aeq alt}, we see that $f \aeq e$.
\end{proof}

Whilst the above shows that $\aeq$ is reflexive and symmetric in the contractive case, it need not be transitive, as we will see for the tiling semigroup. So $\sim$ is the transitive closure of $\aeq$.

\begin{remark}
Lemma \ref{lem: aeq ref sym} implies that, in the contractive case, we may equivalently define $\aeq$ by demanding that the right-infinite tail $w'$ of \eqref{eq: aeq alt} remains constant in $n$, and that each $g_n \in \NN$.
\end{remark}

\begin{definition} \label{def: limit space}
The {\bf limit space} $\limsp$ of a self-similar semigroup action is defined as the quotient space $\mc{F}^- / \sim$. The shift map $\sigma \colon \mc{F}^{-} \to \mc{F}^{-}$, given by $\cdots e_{-3} e_{-2} e_{-1} \mapsto \cdots e_{-4} e_{-3} e_{-2}$ induces a map $\sigma \colon \limsp \to \limsp$. We denote its inverse limit by
\begin{equation}\label{omega_inv_limit}
\Omega \coloneqq \varprojlim (\limsp \xleftarrow{\sigma} \limsp \xleftarrow{\sigma} \limsp \xleftarrow{\sigma} \cdots).
\end{equation}
\end{definition}

We now return to the case of the tiling semigroup $\mc{T}$ acting on $\mc{F} \cong \Omp$. We construct a map
\begin{equation}\label{alpha_map}
\alpha \colon \mc{F}^- \hspace{-0.2cm} \longrightarrow Y \coloneqq \bigsqcup_{p \in \mc{P}} \supp(P),
\end{equation}
where the range of the map is the disjoint union of (supports of) prototiles. Let us recall the following elementary lemma.

\begin{lemma}\label{lem_compact_inclusion}
Let $\cdots \subset S_{-3} \subset S_{-2} \subset S_{-1} $ be a sequence of nested non-empty compact subsets of $\R^d$ such that the corresponding diameters $d_i \coloneqq sup_{x_1, x_2 \in S_i} |x_1-x_2| \to 0$ as $i \to -\infty$. Then $\cap_{i=-1}^{-\infty} S_i$  is a single point in $\R^d$.
\end{lemma}

We will construct such a nested sequence from elements of $\mc{F}^-$. This is done using successively finer partitions of $Y$ under substitution: Given $e = \cdots e_{-2}e_{-1} \in \mc{F}^-$, we let $S_{-1} \coloneqq \supp(s(e_{-1}))$. For $n > 1$, each supertile extension $e_{-n}$ embeds $\supp(r(e_{-n}))$ into $\supp(s(e_{-n}))$ as subtiles of scale $\lambda^{-1}$ of the original size. Thus, letting $S_{-n} \coloneqq \supp(r(e_{-n}))$ be the corresponding subset of $S_{-(n-1)}$, we get a nested sequence of subtile inclusions $\cdots \subset S_{-2} \subset S_{-1}$. By Lemma \ref{lem_compact_inclusion} their intersection is some point $x_{-\infty} \in S_{-n} \subset Y$ for each $n$, and we define a continuous map $\alpha \colon \mc{F}^- \hspace{-0.2cm} \longrightarrow Y$ by $\alpha(e) \coloneqq x_{-\infty}$.

\begin{lemma} \label{thm: aeq}
Suppose $(\mc{T},\mc{F})$ is a self-similar inverse semigroup associated with a recognisable substitution. We have that $e \aeq f$ if and only if there is some tiling $T \in \Omp$, and tiles $t=p+x $ and $t'=q+y$ in $T$ with $p$, $q \in \PP$ and $x$, $y \in \R^d$ such that $\alpha(e) \in \supp(p)$, $\alpha(f) \in \supp(q)$ and $\alpha(e)+x=\alpha(f)+y$. That is, $\alpha(e)$ and $\alpha(f)$ are identical points of a prototile, or points on prototile boundaries which can coincide in adjacent tiles.
\end{lemma}

\begin{proof}
Suppose that such a tiling $T$ exists with $\alpha(e)+x=\alpha(f)+y$ in the supports of $t$ and $t'$. Using the homeomorphism $\tau:\mc{F} \to \Omp$ from \eqref{eq: strings <-> transversal}, let $w_e = \tau^{-1}(T-x(t))$ and $w_f = \tau^{-1}(T-x(t'))$. For each $n \in \N$, consider the tilings $E_n=\tau(e_{-n} \cdots e_{-2}e_{-1}w_e)$ and $F_n=\tau(f_{-n} \cdots f_{-2} f_{-1}w_f)$, respectively. Notice that the sequences of tilings $(E_n)$ and $(F_n)$ are given by successive substitution. Moreover, since $\alpha(e)$ and $\alpha(f)$ correspond to a shared point, we have that the tilings $E_n$ and $F_n$ are equal, up to translating from the origin tile of $E_n$ to the adjacent origin tile of $F_n$. It follows that we may choose semigroup elements $g_n \in \mc{T}$ corresponding to translations between adjacent tiles and so that $g_n \cdot \tau^{-1}(E_n) = \tau^{-1}(F_n)$. This shows that \eqref{eq: aeq alt} is satisfied with $h = [q,P_{pq},p]$ where $P_{pq}$ may be taken as a star-patch with $p$, $q$ meeting analogously to $t$ and $t'$. By FLC, there are only finitely many such patches, so $e \aeq f$.

Conversely, suppose that $e \aeq f$, and take $w$, $(g_n)$ and $h$ as in \eqref{eq: aeq alt}. We have that
\begin{equation} \label{eq: limit space proof}
g_n\cdot (e_{-n} \cdots e_{-3}e_{-2}e_{-1}w) = f_{-n} \cdots f_{-1}(h \cdot w),
\end{equation}
where each $g_n \in \NN$ and $h \in \NN$. Hence, we have tilings $T = \tau(w)$ and $T' = \tau(h \cdot w)$ which are equal up to a translation between adjacent origin tiles $t = t_0 \in T$ and $t' = t'_0 \in T'$. Moreover, for each level $n$ of substitution, the tilings $E_n = \tau(e_{-n} \cdots e_{-2}e_{-1}w)$ and $F_n = \tau(f_{-n} \cdots f_{-2} f_{-1}w)$ remain equal up to translation between adjacent origin tiles $t_n \in \sub(t_{n-1})$ and $t'_n \in \sub(t'_{n-1})$. So we have $T-x(t')=T'$ and, letting $t=p+x$ and $t'=q+y$ for $p$, $q \in \PP$, by definition of $\alpha$ we have that $\alpha(e)+x=\alpha(f)+y$, as required.
\end{proof}

\begin{theorem} \label{thm: limit spaces of sub tilings}
Suppose $(\mc{T},\mc{F})$ is a self-similar inverse semigroup associated with a recognisable substitution $\sub$. The limit space $\limsp$ is homeomorphic to the Anderson--Putnam complex of the substitution,  and the inverse limit $\Omega$ in \eqref{omega_inv_limit} is conjugate to the continuous hull $\Omega_\sub$.
\end{theorem}

\begin{proof}
Let $x \sim_\mathrm{AP} y$ be the the relation on $Y$ that identifies points of prototiles that coincide in some tiling. The Anderson--Putnam complex $\Gamma_0$ \cite{AP} is defined as the quotient of $Y$ under the transitive closure of $\sim_\mathrm{AP}$. Let us denote the quotient map by $q_\mathrm{AP} \colon Y \to \Gamma_0$. We have that $\alpha \colon \mc{F}^- \to Y$ is also a quotient map, since it is a surjective map between compact Hausdorff spaces. By Lemma \ref{thm: aeq}, we have that $e \aeq f$ in $\mc{F}^-$ if and only if $\alpha(e) \sim_\mathrm{AP} \alpha(f)$. It follows that the quotient map $q \colon \mc{F}^- \to \limsp$ may be identified with composition $\alpha \circ q_\mathrm{AP}$ and hence $\limsp$ is homeomorphic to the Anderson--Putnam complex $\Gamma_0$.

Given $e \in \mc{F}^-$, its shift in the limit space may be identified with $q_\mathrm{AP}(\alpha(\sigma (e)))$. By the definition of $\alpha$, the point $\alpha(\sigma (e)) \in Y$ is given by substituting $\alpha(e)$, considered as a point of a prototile in $Y$. Hence $\sigma \colon \limsp \to \limsp$ agrees with the map induced by substitution on the Anderson--Putnam complex, so $\Omega$ is the continuous hull by \cite[Theorem 4.3]{AP}.
\end{proof}

\begin{remark}
Definition \ref{def: limit space} generalises the notion of the limit space and associated inverse limit (the {\bf limit solenoid}) for self-similar groups. In the case of self-similar semigroups, a notion of the limit solenoid has already been defined without use of the intermediary limit space, as a quotient on the bi-infinite Markov shift $\mc{F}_\Z$ by an equivalence relation similar to the asymptotic equivalence relation above \cite[Definition 3.4]{Nek_SSIS}. In the case considered here, there is a natural map $\beta \colon \mc{F}_\Z \to \Omega_\sub$, defined as follows. Given $w = w_- w_+ \in \mc{F}_\Z$, where $w_- = \cdots e_{-2} e_{-1} \in \mc{F}_-$ and $w_+ = e_0 e_1 \cdots \in \mc{F}$, we define $\beta(w)$ to be the tiling $\tau(w_+)$, translated with the origin over the point corresponding to $\alpha(w_-)$ in the origin tile $r(w_+) = s(w_-)$. This defines a quotient map to the tiling space, and it is easy to see that it identifies points of the Markov shift if and only if they correspond to addresses which are adjacent at all levels, that is, they may be related by a finite sequence of elements $g_n \in \NN$.
\end{remark}

\section{examples} \label{sec: examples}

In this section we study several well-known 1- and 2-dimensional examples of tiling semigroups and their self-similar actions.

\begin{figure}
\begin{center}
\begin{tikzpicture}
\begin{scope}[xshift=0cm,yshift=0cm]
\draw[|-|] (1,2) -- node[above,pos=0.5]{$b$} (3,2);
\draw[|-|] (0,0) -- node[below,pos=0.5]{$a$} (2,0);
\draw[|-|] (2,0) -- node[below,pos=0.5]{$d$} (4,0);
\node[vertex] (vert_l) at (1,0) {};
\node[vertex] (vert_r) at (3,0) {}
	edge [->,>=latex,out=135,in=45,thick] node[above,pos=0.5]{$[a,P_{ad},d]$} (vert_l);
\node[vertex] (vert_u) at (2,2)  {}
	edge [->,>=latex,out=225,in=90,thick] node[left,pos=0.5]{$(a,b)$} (vert_l)
	edge [->,>=latex,out=315,in=90,thick] node[right,pos=0.5]{$(d,b)$} (vert_r);
\node at (0,0) {$\ddag$};
\node at (4,0) {$\ddag$};
\node at (1,2) {$\ddag$};
\node at (3,2) {$\ddag$};
\end{scope}

\begin{scope}[xshift=5cm,yshift=0cm]
\draw[|-|] (1,2) -- node[above,pos=0.5]{$c$} (3,2);
\draw[|-|] (0,0) -- node[below,pos=0.5]{$a$} (2,0);
\draw[|-|] (2,0) -- node[below,pos=0.5]{$d$} (4,0);
\node[vertex] (vert_l) at (1,0) {};
\node[vertex] (vert_r) at (3,0) {}
	edge [->,>=latex,out=135,in=45,thick] node[above,pos=0.5]{$[a,P_{ad},d]$} (vert_l);
\node[vertex] (vert_u) at (2,2)  {}
	edge [->,>=latex,out=225,in=90,thick] node[left,pos=0.5]{$(a,c)$} (vert_l)
	edge [->,>=latex,out=315,in=90,thick] node[right,pos=0.5]{$(d,c)$} (vert_r);
\node at (0,0) {$\ddag$};
\node at (4,0) {$\ddag$};
\node at (1,2) {$\ddag$};
\node at (3,2) {$\ddag$};
\end{scope}

\begin{scope}[xshift=10cm,yshift=0cm]
\draw[|-|] (1,2) -- node[above,pos=0.5]{$d$} (3,2);
\draw[|-|] (3,2) -- node[above,pos=0.5]{$b$} (5,2);
\draw[|-|] (0,0) -- node[below,pos=0.5]{$b$} (2,0);
\draw[|-|] (2,0) -- node[below,pos=0.5]{$a$} (4,0);
\draw[|-|] (4,0) -- node[below,pos=0.5]{$d$} (6,0);
\node[vertex] (vert_l) at (1,0) {};
\node[vertex] (vert_r) at (3,0) {}
	edge [->,>=latex,out=135,in=45,thick] node[above,pos=0.5]{$[b,P_{ba},a]$} (vert_l);
\node[vertex] (vert_ul) at (2,2)  {}
	edge [->,>=latex,out=225,in=90,thick] node[left,pos=0.5]{$(b,d)$} (vert_l);
\node[vertex] (vert_ur) at (4,2)  {}
	edge [->,>=latex,out=135,in=45,thick] node[above,pos=0.5]{$[d,P_{db},b]$} (vert_ul)
	edge [->,>=latex,out=225,in=90,thick] node[right,pos=0.5]{$(a,b)$} (vert_r);
\node at (0,0) {$\ddag$};
\node at (2,0) {$\ddag$};
\node at (6,0) {$\ddag$};
\node at (5,2) {$\ddag$};
\node at (3,2) {$\ddag$};
\end{scope}
\end{tikzpicture}
\end{center}
\caption{An illustration of how the first three formulae after \eqref{Fib_auto1} are deduced. The double daggers at tile edges denote supertile boundaries.}
\label{fig:Auto left}
\end{figure}
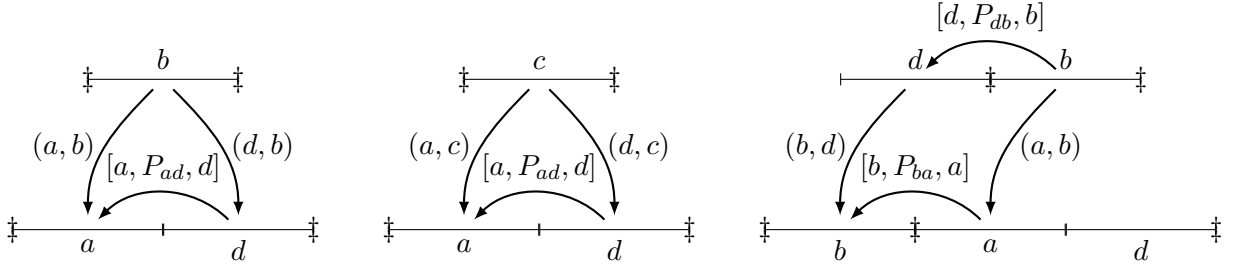

\begin{figure}
\begin{center}
\begin{tikzpicture}
\begin{scope}[xshift=0cm,yshift=0cm]
\draw[|-|] (0,4) -- node[above,pos=0.5]{$a$} (2,4);
\draw[|-|] (2,4) -- node[above,pos=0.5]{$d$} (4,4);
\draw[|-|] (-1,2) -- node[below,pos=0.5]{$a$} (1,2);
\draw[|-|] (1,2) -- node[below,pos=0.5]{$d$} (3,2);
\draw[|-|] (3,2) -- node[below,pos=0.5]{$b$} (5,2);
\draw[|-|] (0,0) -- node[below,pos=0.5]{$b$} (2,0);
\draw[|-|] (2,0) -- node[below,pos=0.5]{$a$} (4,0);
\draw[|-|] (4,0) -- node[below,pos=0.5]{$d$} (6,0);
\node[vertex] (vert_l) at (1,0) {};
\node[vertex] (vert_r) at (3,0) {}
	edge [<-,>=latex,out=135,in=45,thick] node[above,pos=0.5]{$[a,P_{ba},b]$} (vert_l);
\node[vertex] (vert_ul) at (2,2)  {}
	edge [->,>=latex,out=225,in=90,thick] node[left,pos=0.5]{$(b,d)$} (vert_l);
\node[vertex] (vert_ur) at (4,2)  {}
	edge [<-,>=latex,out=135,in=45,thick] node[above,pos=0.5]{$[b,P_{db},d]$} (vert_ul)
	edge [->,>=latex,out=225,in=90,thick] node[right,pos=0.5]{$(a,b)$} (vert_r);
\node[vertex] (vert_uul) at (1,4)  {}
	edge [->,>=latex,out=315,in=90,thick] node[left,pos=0.5]{$(d,a)$} (vert_ul);
\node[vertex] (vert_uur) at (3,4)  {}
	edge [->,>=latex,out=315,in=90,thick] node[right,pos=0.5]{$(b,d)$} (vert_ur);
\node at (0,0) {$\ddag$};
\node at (2,0) {$\ddag$};
\node at (-1,2) {$\ddag$};
\node at (6,0) {$\ddag$};
\node at (5,2) {$\ddag$};
\node at (3,2) {$\ddag$};
\node at (0,4) {$\ddag$};
\node at (4,4) {$\ddag$};
\end{scope}

\begin{scope}[xshift=8cm,yshift=0cm]
\draw[|-|] (0,4) -- node[above,pos=0.5]{$b$} (2,4);
\draw[|-|] (2,4) -- node[above,pos=0.5]{$a$} (4,4);
\draw[|-|] (-2,2) -- node[above,pos=0.5]{$a$} (0,2);
\draw[|-|] (0,2) -- node[above,pos=0.4]{$d$} (2,2);
\draw[|-|] (2,2) -- node[above,pos=0.6]{$c$} (4,2);
\draw[|-|] (4,2) -- node[above,pos=0.5]{$d$} (6,2);
\draw[|-|] (0,0) -- node[below,pos=0.5]{$b$} (2,0);
\draw[|-|] (2,0) -- node[below,pos=0.5]{$a$} (4,0);
\draw[|-|] (4,0) -- node[below,pos=0.5]{$d$} (6,0);
\node[vertex] (vert_l) at (1,0) {};
\node[vertex] (vert_r) at (3,0) {}
	edge [<-,>=latex,out=135,in=45,thick] node[above,pos=0.5]{$[a,P_{ba},b]$} (vert_l);
\node[vertex] (vert_ul) at (1,2)  {}
	edge [->,>=latex,out=270,in=90,thick] node[left,pos=0.5]{$(b,d)$} (vert_l);
\node[vertex] (vert_ur) at (3,2)  {}
	edge [<-,>=latex,out=135,in=45,thick] node[above,pos=0.5]{$[c,P_{dc},d]$} (vert_ul)
	edge [->,>=latex,out=270,in=90,thick] node[right,pos=0.5]{$(a,c)$} (vert_r);
\node[vertex] (vert_uul) at (1,4)  {}
	edge [->,>=latex,out=270,in=90,thick] node[left,pos=0.5]{$(d,b)$} (vert_ul);
\node[vertex] (vert_uur) at (3,4)  {}
	edge [->,>=latex,out=270,in=90,thick] node[right,pos=0.5]{$(c,a)$} (vert_ur);
\node at (0,0) {$\ddag$};
\node at (2,0) {$\ddag$};
\node at (-2,2) {$\ddag$};
\node at (6,0) {$\ddag$};
\node at (6,2) {$\ddag$};
\node at (2,2) {$\ddag$};
\node at (0,4) {$\ddag$};
\node at (4,4) {$\ddag$};
\end{scope}
\end{tikzpicture}
\end{center}
\caption{An illustration of how the first three formulae after \eqref{Fib_auto2} are deduced. The double daggers at tile edges denote supertile boundaries.}
\label{fig:Auto right}
\end{figure}
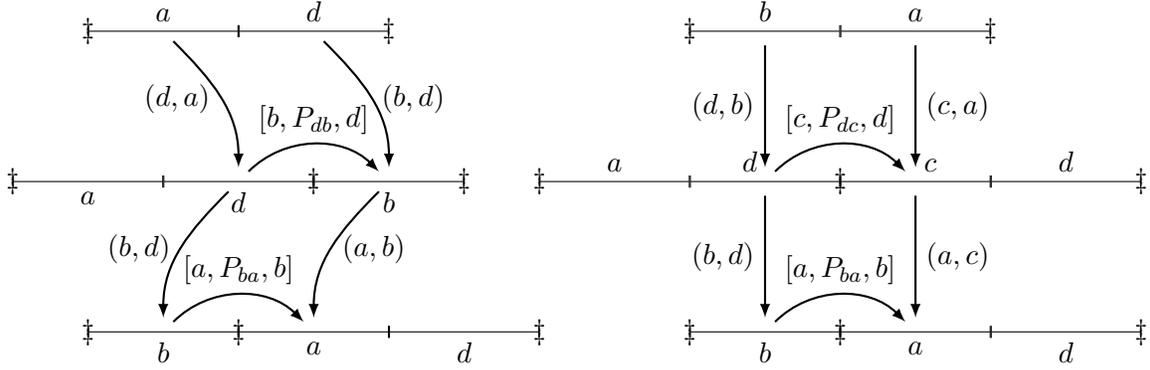

\begin{example}\label{ex: fib rules}
We return to the border forcing Fibonacci tiling of Examples \ref{ex: fib} and \ref{geometric intuition}. The self-similar inverse semigroup is generated by doubly pointed patches consisting of all possible single tile patches and connected pair patches appearing anywhere in a Fibonacci tiling. Note that Figures \ref{fig:Auto left} and \ref{fig:Auto right} show how we geometrically deduce the self-similar relation on a selection of generating elements. The following doubly pointed patches, represented here along with their self-similar action, generate the semigroup of the Fibonacci tiling.

	\begingroup
		\allowdisplaybreaks
\begin{align}
\label{Fib_auto1}
[a,P_{ad},d] \cdot (d,b)w&=(a,b)\ [b,P_b,b] \cdot w; \\
\notag
[a,P_{ad},d] \cdot (d,c)w&=(a,c) \ [c,P_c,c] \cdot w; \\
\notag
[b,P_{ba},a] \cdot (a,b)w&=(b,d) \ [d,P_{db},b] \cdot w;  \\
\notag
[b,P_{ba},a] \cdot (a,c)w&=(b,d)\ [d,P_{dc},c] \cdot w; \\
\notag
[c,P_{cd},d] \cdot (d,a)w&=(c,a)\ [a,P_{a},a] \cdot w; \\
\notag
[d,P_{db},b] \cdot (b,d)(d,a)w&=(d,c)\ [c,P_{cd},d] \cdot (d,a)w; \\
\notag
[d,P_{db},b] \cdot (b,d)(d,b)w&=(d,a)\ [a,P_{ad},d] \cdot (d,b)w; \\
\notag
[d,P_{db},b] \cdot (b,d)(d,c)w&=(d,a)\ [a,P_{ad},d] \cdot (d,c)w; \\
\notag
[d,P_{dc},c] \cdot (c,a)w&=(d,b)\ [b,P_{ba},a] \cdot w; \\
\label{Fib_auto2}
[a,P_{ba},b] \cdot (b,d)(d,a)w&=(a,b)\ [b,P_{db},d] \cdot (d,a)w; \\
\notag
[a,P_{ba},b] \cdot (b,d)(d,b)w&=(a,c)\ [c,P_{dc},d] \cdot (d,b)w; \\
\notag
[a,P_{ba},b] \cdot (b,d)(d,c)w&=(a,b)\ [b,P_{db},d] \cdot (d,c)w; \\
\notag
[b,P_{db},d] \cdot (d,a)w&=(b,d)\ [d,P_{ad},a] \cdot w; \\
\notag
[b,P_{db},d] \cdot (d,c)w&=(b,d)\ [d,P_{cd},c] \cdot w; \\
\notag
[c,P_{dc},d] \cdot (d,b)w&=(c,a)\ [a,P_{ba},b] \cdot w; \\
\notag
[d,P_{ad},a] \cdot (a,b)w&=(d,b)\ [b,P_{b},b] \cdot w; \\
\notag
[d,P_{ad},a] \cdot (a,c)w&=(d,c)\ [c,P_{c},c] \cdot w; \\
\notag
[d,P_{cd},c] \cdot (c,a)w&=(d,a)\ [a,P_{a},a] \cdot w. 
\end{align}
\qed
\end{example}
	\endgroup

\begin{example}
The simplest border-forcing 2-dimensional example comes from the half-hex tiling. We note that there are six prototiles $\{p_0,p_1,p_2,p_3,p_4,p_5\}$, where the subscript denotes the number of rotations of $p_0$ by $\pi/3$. Similarly, the substitution of each prototile is equivalent up to rotations by $n\pi/3$, see Figure \ref{proto}. Thus, always taking addition to be mod 6, the tile inclusions can be written as:
\[
\{(p_i,p_i), (p_{i+2},p_i), (p_{i+3},p_i), (p_{i+4},p_i) \mid i=0,1,2,3,4,5\}.
\]

\begin{figure}
\begin{center}
\begin{tikzpicture}
\begin{scope}[yshift=1cm]
\HHex{0}{1.73}{180}{0}; 
\HHex{1.5}{0.866}{120}{0}; 
\HHex{-1.5}{0.866}{240}{0}; 
\HHex{0}{-1.73}{0}{0}; 
\HHex{1.5}{-0.866}{60}{0}; 
\HHex{-1.5}{-0.866}{-60}{0}; 
\node at (30:1.3) {$p_2$};
\node at (90:1.3) {$p_3$};
\node at (150:1.3) {$p_4$};
\node at (210:1.3) {$p_5$};
\node at (270:1.3) {$p_0$};
\node at (330:1.3) {$p_1$};
\end{scope}
\begin{scope}[xshift=10cm]
\HHex{-4.5}{0.5}{0}{0};
\node at (-4.5,0.85) {$p_0$};
\HHexI{0}{0}{0}{-1};
\node at (90:0.35) {$p_0$};
\node at (30:1.3) {$p_2$};
\node at (90:1.3) {$p_3$};
\node at (150:1.3) {$p_4$};
\draw[->, thick] (-3.25,1) -- node[above] {$\varphi$} (-2.15,1);
\end{scope}
\end{tikzpicture}
\end{center}
\caption{The half-hex prototiles are on the left and the substitution of $p_0$ is on the right. All other substitutions are rotations of $p_0$ by $n\pi/3$.}
\label{proto}
\end{figure}
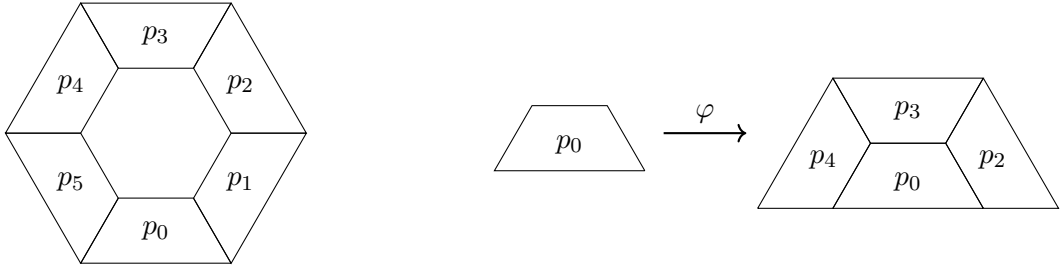

The substitution of $p_0$ appears in Figure \ref{proto}. The graph associated with this substitution appears in Figure \ref{HH_sub}.

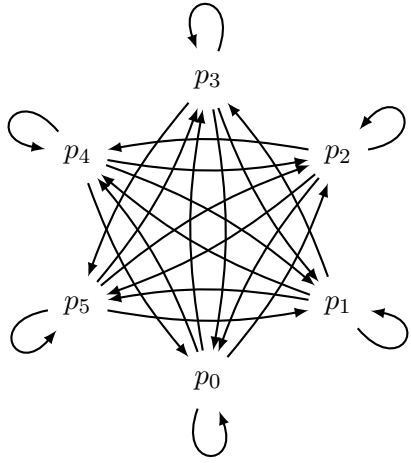
\begin{figure}
\begin{center}
\begin{tikzpicture}
\node[vertex] (vert_p0) at (270:2) {$p_0$}
	edge [->,>=latex,out=250,in=290,thick,loop] node[below,pos=0.5]{} (vert_p0);
\node[vertex] (vert_p1) at (330:2) {$p_1$}
	edge [->,>=latex,out=310,in=350,thick,loop] node[right,pos=0.5]{} (vert_p1);
\node[vertex] (vert_p2) at (30:2) {$p_2$}
	edge [->,>=latex,out=10,in=50,thick,loop] node[right,pos=0.5]{} (vert_p2)
	edge [<-,>=latex,out=250,in=50,thick] node[right,pos=0.5]{} (vert_p0)
	edge [->,>=latex,out=230,in=70,thick] node[right,pos=0.5]{} (vert_p0);
\node[vertex] (vert_p3) at (90:2) {$p_3$}
	edge [->,>=latex,out=70,in=110,thick,loop] node[above,pos=0.5]{} (vert_p3)
	edge [<-,>=latex,out=260,in=100,thick] node[right,pos=0.5]{} (vert_p0)
	edge [->,>=latex,out=280,in=80,thick] node[right,pos=0.5]{} (vert_p0)
	edge [<-,>=latex,out=250+60,in=50+60,thick] node[right,pos=0.5]{} (vert_p1)
	edge [->,>=latex,out=230+60,in=70+60,thick] node[right,pos=0.5]{} (vert_p1);
\node[vertex] (vert_p4) at (150:2) {$p_4$}
	edge [->,>=latex,out=130,in=170,thick,loop] node[left,pos=0.5]{} (vert_p4)
	edge [<-,>=latex,out=260+60,in=100+60,thick] node[right,pos=0.5]{} (vert_p1)
	edge [->,>=latex,out=280+60,in=80+60,thick] node[right,pos=0.5]{} (vert_p1)
	edge [<-,>=latex,out=250+60,in=50+60,thick] node[right,pos=0.5]{} (vert_p0)
	edge [->,>=latex,out=230+60,in=70+60,thick] node[right,pos=0.5]{} (vert_p0)
	edge [<-,>=latex,out=250+120,in=50+120,thick] node[right,pos=0.5]{} (vert_p2)
	edge [->,>=latex,out=230+120,in=70+120,thick] node[right,pos=0.5]{} (vert_p2);
\node[vertex] (vert_p5) at (210:2) {$p_5$}
	edge [->,>=latex,out=190,in=230,thick,loop] node[left,pos=0.5]{} (vert_p5)
	edge [<-,>=latex,out=260+120,in=100+120,thick] node[right,pos=0.5]{} (vert_p2)
	edge [->,>=latex,out=280+120,in=80+120,thick] node[right,pos=0.5]{} (vert_p2)
	edge [<-,>=latex,out=10,in=170,thick] node[right,pos=0.5]{} (vert_p1)
	edge [->,>=latex,out=-10,in=190,thick] node[right,pos=0.5]{} (vert_p1)
	edge [<-,>=latex,out=250+180,in=50+180,thick] node[right,pos=0.5]{} (vert_p3)
	edge [->,>=latex,out=230+180,in=70+180,thick] node[right,pos=0.5]{} (vert_p3);
\end{tikzpicture}	
\end{center}
\caption{The substitution graph of the half-hex tiling.}
\label{HH_sub}
\end{figure}

The self-similar inverse semigroup is generated by the collection of doubly pointed patches consisting of all single tile patches and connected pair patches appearing in a half-hex tiling. All such tile pairs appear in Figure \ref{automaton_elements} up to rotation. For each connected pair of tiles, the generating doubly pointed patch $[q,X,p]$ represents a translation across an edge of tile $p$ where $X \in \{A,B,C,D\}$ represents the two-tile patch connected across one of the 4 edges of $p$. For $p_0$, we set edges $A$--$D$ to be the 4 edges starting from the bottom and rotating counterclockwise. See Figure \ref{automaton_elements} for complete clarity. 

\begin{figure}
\begin{center}
\begin{tikzpicture}
\begin{scope}[xshift=0cm,yshift=0cm]
\HHex{0}{0}{0}{-1}; 
\HHex{0}{0}{180}{-1}; 
\node[vertex] (vert_p) at (90:1.1) {$(p_i,x)$};
\node[vertex] (vert_q) at (-90:1.1) {$(p_{i+3},x)$}
	edge [<-,>=latex,out=132,in=228,thick] node[right,pos=0.55]{$\scriptstyle [p_{i+3},A,p_i]$} (vert_p);
\end{scope}

\begin{scope}[xshift=4.2cm,yshift=-1.8cm]
\HHex{0}{0}{0}{-1}; 
\HHex{3}{1.732}{60}{-1}; 
\node[vertex] (vert_p) at (90:0.6) {$(p_i,x)$};
\node[vertex] (vert_q) at (40:3.1) {$(p_{i+1},x)$}
	edge [<-,>=latex,out=220,in=20,thick] node[left,pos=0.45]{$\scriptstyle [p_{i+1},B,p_i]$} (vert_p);
\end{scope}

\begin{scope}[xshift=9cm,yshift=-1.8cm]
\HHex{0}{0}{0}{-1}; 
\HHex{3}{1.732}{120}{-1}; 
\node[vertex] (vert_p) at (90:0.6) {$(p_i,x)$};
\node[vertex] (vert_q) at (35:2.8) {$(p_{i+2},x)$}
	edge [<-,>=latex,out=220,in=0,thick] node[left,pos=0.15]{$\scriptstyle [p_{i+2},B,p_i]$} (vert_p);
\end{scope}

\begin{scope}[xshift=0cm,yshift=-7.5cm]
\HHex{0}{0}{0}{-1}; 
\HHex{0}{3.464}{180}{-1}; 
\node[vertex] (vert_p) at (90:0.6) {$(p_i,x)$};
\node[vertex] (vert_q) at (90:2.8) {$(p_{i+3},x)$}
	edge [<-,>=latex,out=228,in=132,thick] node[right,pos=0.6]{$\scriptstyle [p_{i+3},C,p_i]$} (vert_p);
\end{scope}

\begin{scope}[xshift=4.7cm,yshift=-7.5cm]
\HHex{0}{0}{0}{-1}; 
\HHex{0}{3.464}{240}{-1}; 
\node[vertex] (vert_p) at (90:0.6) {$(p_i,x)$};
\node[vertex] (vert_q) at (80:3.2) {$(p_{i+4},x)$}
	edge [<-,>=latex,out=228,in=132,thick] node[right,pos=0.65]{$\scriptstyle [p_{i+4},C,p_i]$} (vert_p);
\end{scope}

\begin{scope}[xshift=10cm,yshift=-7.5cm]
\HHex{0}{0}{0}{-1}; 
\HHex{0}{3.464}{120}{-1}; 
\node[vertex] (vert_p) at (90:0.6) {$(p_i,x)$};
\node[vertex] (vert_q) at (100:3.2) {$(p_{i+2},x)$}
	edge [<-,>=latex,out=312,in=48,thick] node[left,pos=0.65]{$\scriptstyle [p_{i+2},C,p_i]$} (vert_p);
\end{scope}

\begin{scope}[xshift=3cm,yshift=-12cm]
\HHex{0}{0}{0}{-1}; 
\HHex{-3}{1.732}{300}{-1}; 
\node[vertex] (vert_p) at (90:0.6) {$(p_i,x)$};
\node[vertex] (vert_q) at (140:3.2) {$(p_{i+5},x)$}
	edge [<-,>=latex,out=320,in=160,thick] node[right,pos=0.47]{$\scriptstyle [p_{i+5},D,p_i]$} (vert_p);
\end{scope}

\begin{scope}[xshift=10cm,yshift=-12cm]
\HHex{0}{0}{0}{-1}; 
\HHex{-3}{1.732}{-120}{-1}; 
\node[vertex] (vert_p) at (90:0.6) {$(p_i,x)$};
\node[vertex] (vert_q) at (145:2.8) {$(p_{i+4},x)$}
	edge [<-,>=latex,out=320,in=180,thick] node[right,pos=0.10]{$\, \scriptstyle [p_{i+4},D,p_i]$} (vert_p);
\end{scope}

\end{tikzpicture}
\end{center}
\caption{The possible two-tile patches with respect to reference tile $p_i$.}
\label{automaton_elements}
\end{figure}
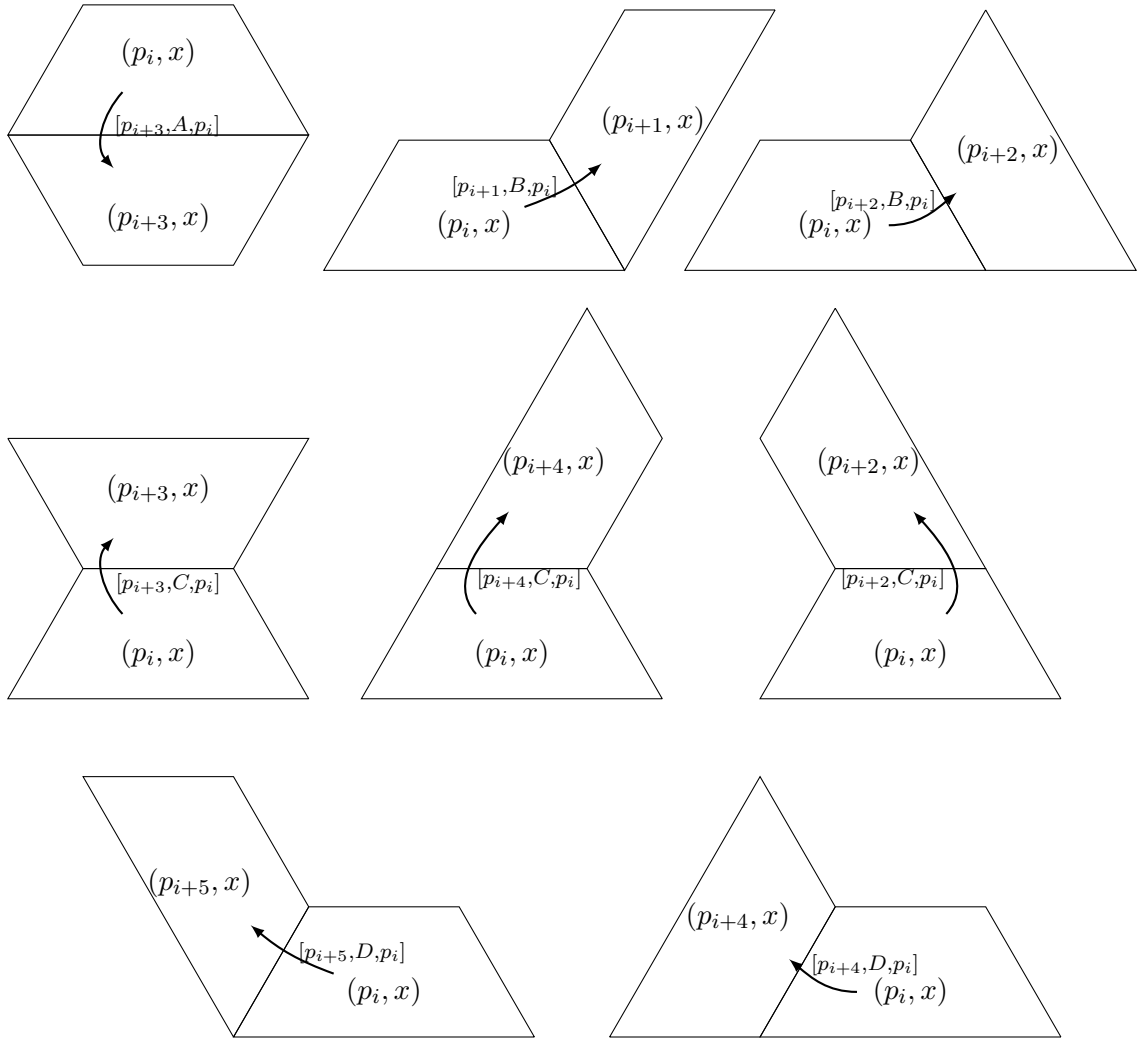

We begin by describing the self-similar relation for the generating doubly pointed patches across the long edge of tile $p_i$, labelled $A$. Note that all subscripts are treated mod$\ 6$ and $w \in \mc{F}$.

	\begingroup
		\allowdisplaybreaks
\begin{align}
\label{HH_auto_1}
[p_{i+3},A,p_i] \cdot (p_i,p_i)w&=(p_{i+3},p_{i+3}) \ [p_{i+3},A,p_i] \cdot w; \\
\notag
[p_{i+3},A,p_i] \cdot (p_i,p_{i+2})(p_{i+2},p_{i+5})w&=(p_{i+3},p_{i+1}) \ [p_{i+1},D,p_{i+2}] \cdot (p_{i+2},p_{i+5})w; \\
\notag
[p_{i+3},A,p_i] \cdot (p_i,p_{i+2})(p_{i+2},p_{i+2})w&=(p_{i+3},p_{i}) \ [p_{i},D,p_{i+2}] \cdot (p_{i+2},p_{i+2})w; \\
\notag
[p_{i+3},A,p_i] \cdot (p_i,p_{i+3})(p_{i+3},p_{i+5})w&=(p_{i+3},p_{i+5}) \ [p_{i+5},C,p_{i+3}] \cdot (p_{i+3},p_{i+5})w;\\
\notag
[p_{i+3},A,p_i] \cdot (p_i,p_{i+3})(p_{i+3},p_{i+1})w&=(p_{i+3},p_{i+1}) \ [p_{i+1},C,p_{i+3}] \cdot (p_{i+3},p_{i+1})w;\\
\notag
[p_{i+3},A,p_i] \cdot (p_i,p_{i+3})(p_{i+3},p_{i})w&=(p_{i+3},p_{i}) \ [p_{i},C,p_{i+3}] \cdot (p_{i+3},p_{i})w;\\
\notag
[p_{i+3},A,p_i] \cdot (p_i,p_{i+4})(p_{i+4},p_{i+1})w&=(p_{i+3},p_{i+5}) \ [p_{i+5},B,p_{i+4}] \cdot (p_{i+4},p_{i+1})w;\\
\notag
[p_{i+3},A,p_i] \cdot (p_i,p_{i+4})(p_{i+4},p_{i+4})w&=(p_{i+3},p_{i}) \ [p_{i},B,p_{i+4}] \cdot (p_{i+4},p_{i+4})w.
\end{align}
	\endgroup

In order to geometrically understand these relations, we illustrate the first two self-similar relations from \eqref{HH_auto_1} in Figures \ref{fig:HH_auto_1} and \ref{fig:HH_auto_2}.

\begin{figure}
\begin{center}
\begin{tikzpicture}
\begin{scope}[xshift=0cm,yshift=0cm]
\HHexI{0}{0}{0}{-1}; 
\HHexI{0}{0}{180}{-1}; 
\node (vert_p) at (90:0.6) {$(p_0,p_0)$};
\node (vert_q) at (270:0.6) {$(p_{3},p_3)$}
	edge [<-,>=latex,out=132,in=228,thick] node[right,pos=0.75]{$\scriptstyle [p_{3},A,p_0]$} (vert_p);
\end{scope}
\begin{scope}[xshift=6cm,yshift=0cm]
\HHex{0}{0}{0}{-1}; 
\HHex{0}{0}{180}{-1}; 
\node[vertex] (vert_p) at (90:1.1) {$(p_0,x)$};
\node[vertex] (vert_q) at (270:1.1) {$(p_{3},x)$}
	edge [<-,>=latex,out=132,in=228,thick] node[right,pos=0.65]{$\scriptstyle [p_{3},A,p_0]$} (vert_p);
\end{scope}
\end{tikzpicture}
\end{center}
\caption{An illustration of the first formula $[p_{3},A,p_0] \cdot (p_0,p_0)w=(p_{3},p_{3}) \ [p_{3},A,p_0] \cdot w$ in \eqref{HH_auto_1} with $i=0$. The left hand side shows the action $[p_{3},A,p_0] \cdot (p_0,p_0)=(p_3,p_3)$ and the right hand side shows that the element $[p_{3},A,p_0]$ comes from the relationship between 1-supertiles.}
\label{fig:HH_auto_1}
\end{figure}

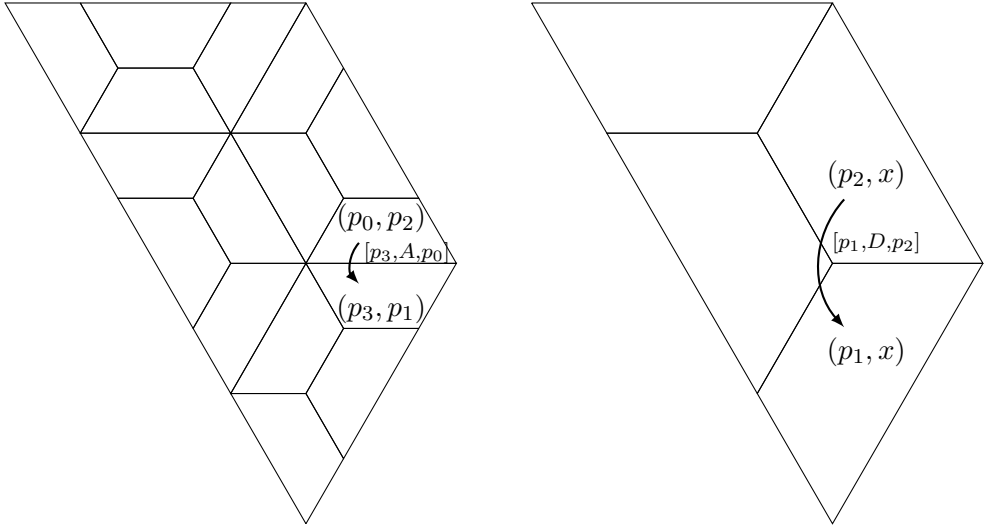
\begin{figure}
\begin{center}
\begin{tikzpicture}
\begin{scope}[xshift=0cm,yshift=0cm]
\HHexII{-3}{0}{-60}{-2}; 
\node (vert_p) at (90:0.6) {$(p_0,p_2)$};
\node (vert_q) at (270:0.6) {$(p_{3},p_1)$}
	edge [<-,>=latex,out=132,in=228,thick] node[right,pos=0.75]{$\scriptstyle [p_{3},A,p_0]$} (vert_p);
\end{scope}
\begin{scope}[xshift=7cm,yshift=0cm]
\HHexI{-3}{0}{-60}{-2}; 
\node (vert_p) at (115:1.3) {$(p_2,x)$};
\node (vert_q) at (245:1.3) {$(p_{1},x)$}
	edge [<-,>=latex,out=132,in=228,thick] node[right,pos=0.65]{$\scriptstyle [p_{1},D,p_2]$} (vert_p);
\end{scope}
\end{tikzpicture}
\end{center}
\caption{An illustration of the second formula $[p_{3},A,p_0] \cdot (p_0,p_{2})(p_{2},p_{5})w=(p_{3},p_{1}) \ [p_{1},D,p_{2}] \cdot (p_{2},p_{5})w$ in \eqref{HH_auto_1} with $i=0$. The left hand side shows the action $[p_{3},A,p_0] \cdot (p_0,p_2)=(p_3,p_1)$ and the right hand side shows that the element $[p_{1},D,p_{2}]$ comes from the relationship between 1-supertiles.}
\label{fig:HH_auto_2}
\end{figure}

We now describe the automaton elements $b_i$, $c_i$ and $d_i$ across the shorter edges of tile $p_i$. Again, we note that all subscripts are treated mod$\ 6$ and $w \in \mc{F}$.

	\begingroup
		\allowdisplaybreaks
\begin{align*}
[p_{i+1},B,p_i] \cdot (p_i,p_{i+2})w&=(p_{i+1},p_{i+5}) \ [p_{i+5},A,p_{i+2}] \cdot w; \\
[p_{i+1},B,p_i] \cdot (p_i,p_{i+3})w&=(p_{i+1},p_{i+3}) w; \\
[p_{i+1},B,p_i] \cdot (p_i,p_{i+4})w&=(p_{i+1},p_{i+4}) w; \\
[p_{i+2},B,p_i] \cdot (p_i,p_i)w&=(p_{i+2},p_{i}) w; \\
[p_{i+2},C,p_i] \cdot (p_i,p_{i+2})w&=(p_{i+2},p_{i+2}) w; \\
[p_{i+3},C,p_i] \cdot (p_i,p_i)w&=(p_{i+3},p_{i}) w; \\
[p_{i+3},C,p_i] \cdot (p_i,p_{i+3})w&=(p_{i+3},p_{i+3}) w; \\
[p_{i+4},C,p_i] \cdot (p_i,p_{i+4})w&=(p_{i+4},p_{i+4}) w; \\
[p_{i+4},D,p_i] \cdot (p_i,p_i)w&=(p_{i+4},p_{i}) w; \\
[p_{i+5},D,p_i] \cdot (p_i,p_{i+2})w&=(p_{i+5},p_{i+2}) w; \\
[p_{i+5},D,p_i] \cdot (p_i,p_{i+3})w&=(p_{i+5},p_{i+3}) w; \\
[p_{i+5},D,p_i] \cdot (p_i,p_{i+4})w&=(p_{i+5},p_{i+1}) \ [p_{i+1},A,p_{i+4}] \cdot w. 
\end{align*}
	\endgroup

Let us note that some relations here could have been omitted by also exploiting the rotational equivariance of substitution. We shall make use of this in the following example. \qed

\end{example}

\begin{figure}
\begin{center}
\includegraphics[scale=0.07]{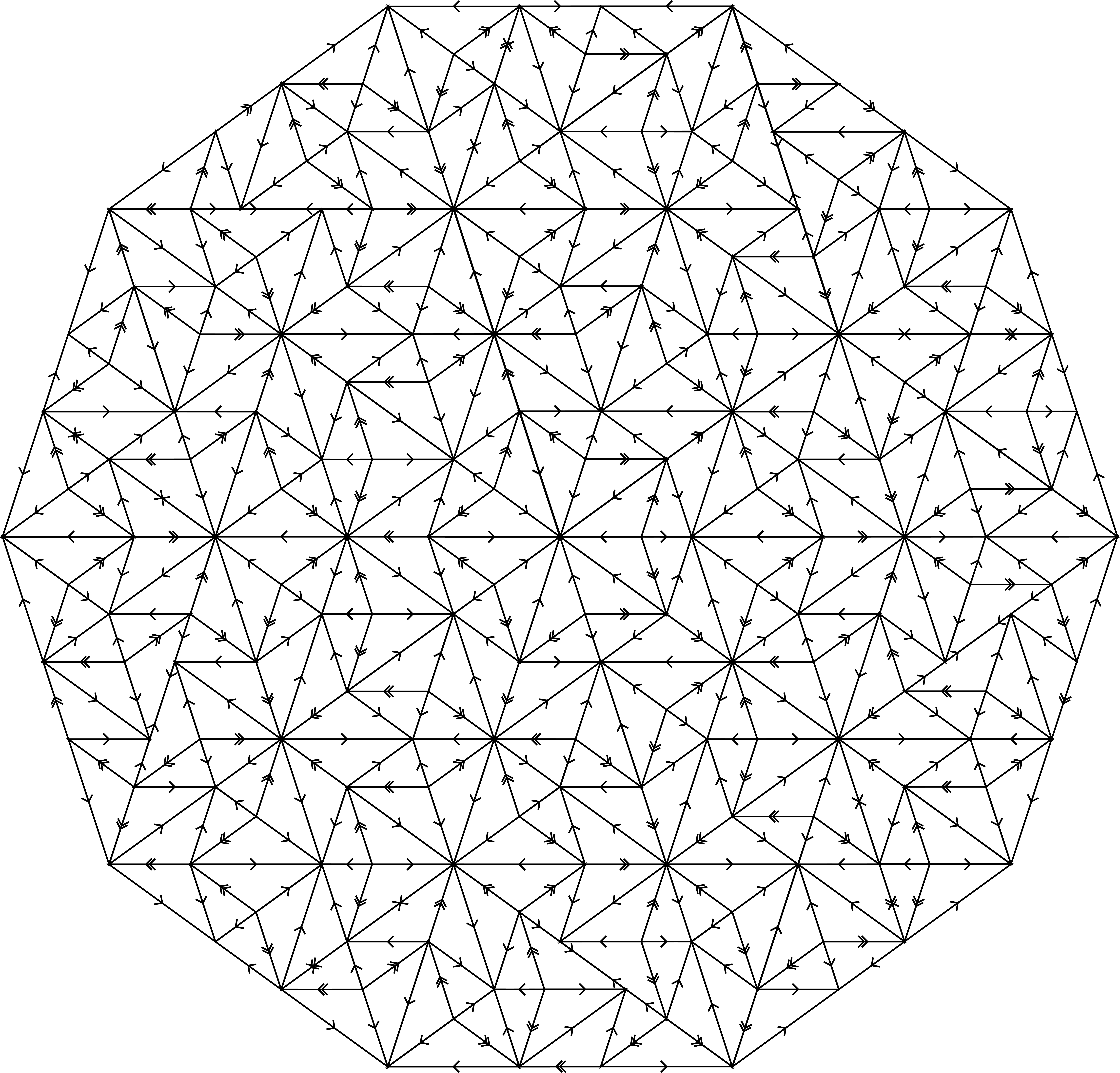}
\end{center}
\caption{A patch of a Penrose tiling.}
\label{Penrose_patch}
\end{figure}

\begin{example}
The most well-known 2-dimensional example was given by Penrose \cite{pentaplexity}, represented here as Robinson triangles. We note that there are forty prototiles $\{a_i,b_i,ra_i,rb_i \mid i=0, \ldots,9\}$, where the subscript denotes the number of rotations by $\pi/5$. By $ra_i$, we mean the reflection of the tile $a_i$ across the vertical, followed by rotation by $i\pi/5$, and analogously for $rb_i$ (we emphasise that we reflect the tile $a_0$ first, and then rotate). The substitutions of $a_0$ and $b_0$ appear in Figure \ref{Penrose_proto}. We have not attempted to display the graph of the substitution. All other prototiles are rigid motions of these and substitution on them is determined by equivariance of the substitution $\sub$. For example, we have that $\sub(ra_4) = \sub(\theta_4(\tau a_0)) = \theta_4 \circ \tau(\sub(a_0))$, where $\theta_4$ is rotation by $4\pi/5$ and $\tau$ is reflection across the vertical. Thus, always taking addition to be (mod$\ 10$), the tile inclusions can be written as
\begin{align*}
(a_7,a_0),(b_3,a_0),(rb_0,b_0),(ra_6,b_0),(b_4,b_0)
\end{align*}
along with the required rigid motions of the above (thus there are $20 \times 5 = 100$ in total). A patch of the Penrose tiling appears in Figure \ref{Penrose_patch}.

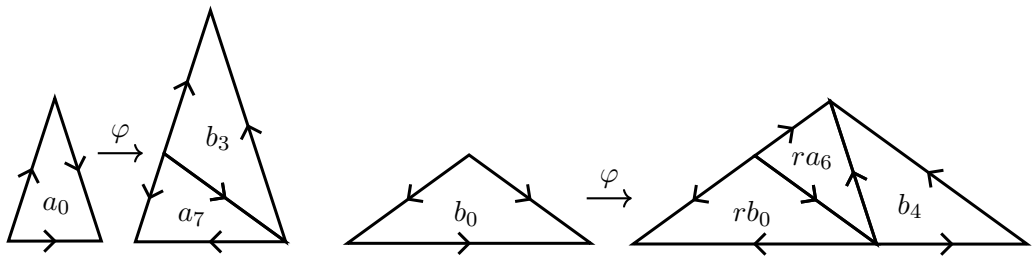
\begin{figure}
\begin{center}
\begin{tikzpicture}
\node at (0,1.8) {\includegraphics[scale=0.2]{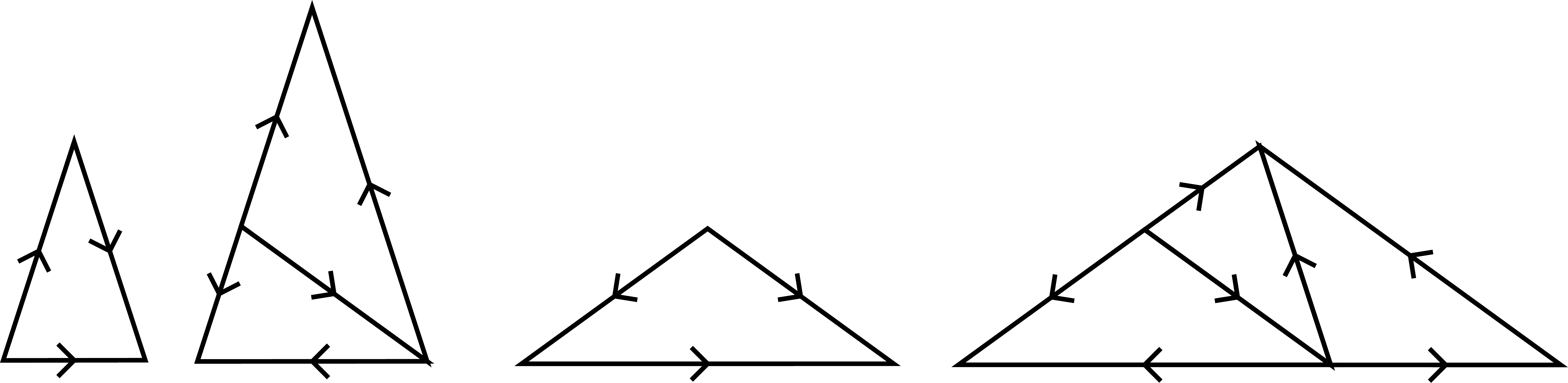}};
\node at (-6.15,0.8) {$a_0$};
\draw[->,thick] (-5.6,1.5) -- node[above] {$\varphi$} (-5.0,1.5);
\node at (-4.35,0.65) {$a_7$};
\node at (-4.0,1.7) {$b_3$};
\node at (-0.7,0.7) {$b_0$};
\draw[->,thick] (0.9,0.9) -- node[above] {$\varphi$} (1.5,0.9);
\node at (3.1,0.7) {$rb_0$};
\node at (5.2,0.8) {$b_4$};
\node at (3.9,1.4) {$ra_6$};
\end{tikzpicture}
\end{center}
\caption{Penrose substitution.}
\label{Penrose_proto}
\end{figure}

\begin{figure}
\begin{center}
\begin{tikzpicture}
\node at (0,0) {\includegraphics[scale=0.21]{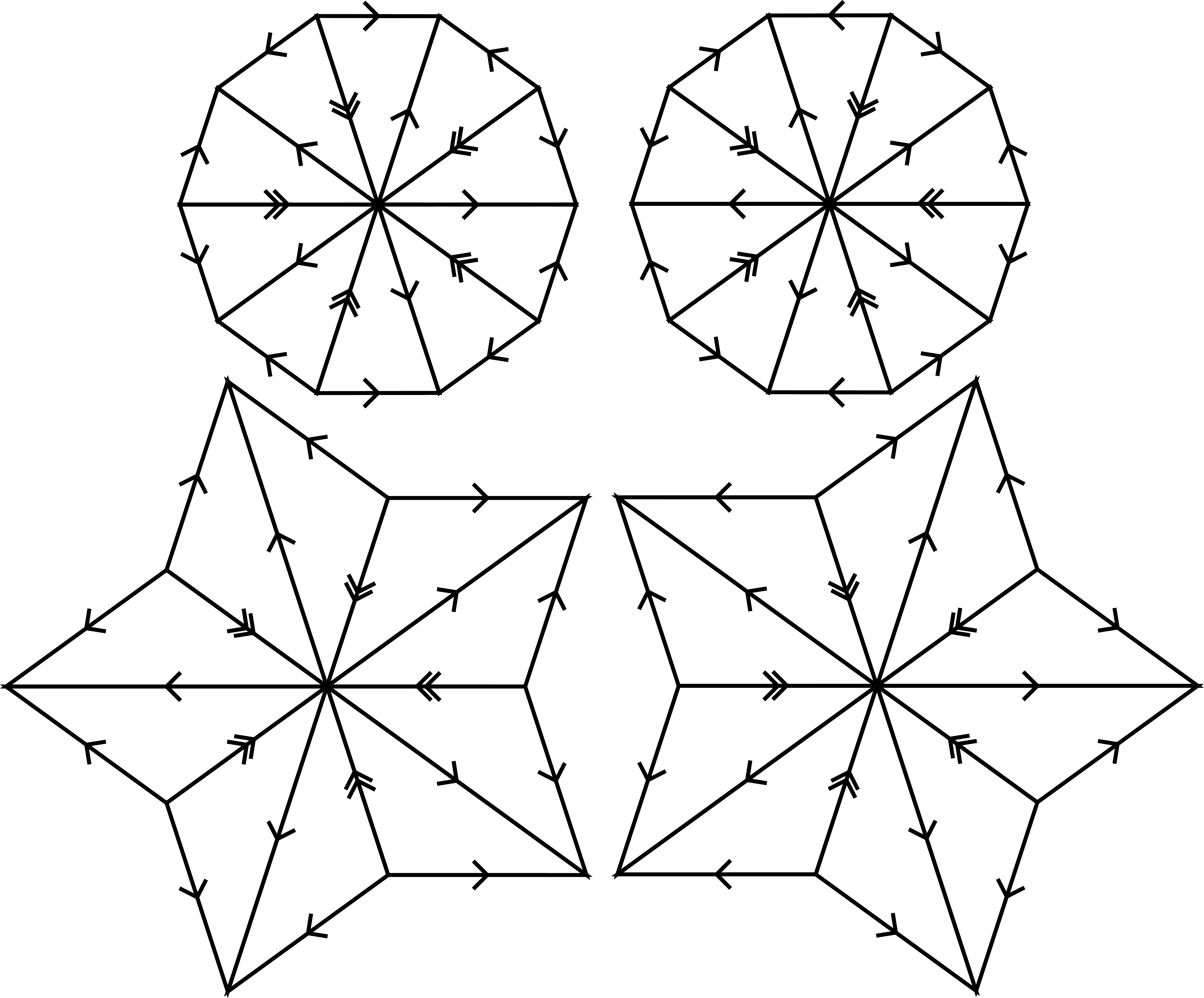}};
\node at (-2.35,1.8) {$a_0$};
\node at (-1.0,2.7) {$a_2$};
\node at (-1.5,4.3) {$a_4$};
\node at (-3.25,4.3) {$a_6$};
\node at (-3.75,2.7) {$a_8$};
\node at (-1.55,2.0) {$ra_1$};
\node at (-1.0,3.5) {$ra_3$};
\node at (-2.35,4.5) {$ra_5$};
\node at (-3.75,3.5) {$ra_7$};
\node at (-3.25,2) {$ra_9$};
\node at (-2.35+4.75,1.8) {$ra_0$};
\node at (-1.0+4.75,2.7) {$ra_2$};
\node at (-1.5+4.75,4.3) {$ra_4$};
\node at (-3.25+4.75,4.3) {$ra_6$};
\node at (-3.75+4.75,2.7) {$ra_8$};
\node at (-1.55+4.75,2.0) {$a_1$};
\node at (-1.0+4.75,3.5) {$a_3$};
\node at (-2.35+4.75,4.5) {$a_5$};
\node at (-3.75+4.75,3.5) {$a_7$};
\node at (-3.25+4.75,2) {$a_9$};
\node at (-4.6,-1.5) {$rb_0$};
\node at (-3.9,-3.4) {$rb_2$};
\node at (-1.8,-3.3) {$rb_4$};
\node at (-1.3,-1.4) {$rb_6$};
\node at (-2.9,-0.3) {$rb_8$};
\node at (-4.6,-2.5) {$b_5$};
\node at (-2.95,-3.6) {$b_7$};
\node at (-1.3,-2.5) {$b_9$};
\node at (-1.8,-0.6) {$b_1$};
\node at (-3.9,-0.6) {$b_3$};
\node at (-4.6+6,-1.5) {$b_4$};
\node at (-4.1+6,-3.3) {$b_6$};
\node at (-2.13+6,-3.3) {$b_8$};
\node at (-1.4+6,-1.55) {$b_0$};
\node at (-3.05+6,-0.3) {$b_2$};
\node at (-4.6+6,-2.5) {$rb_1$};
\node at (-3.05+6,-3.6) {$rb_3$};
\node at (-1.4+6,-2.4) {$rb_5$};
\node at (-2.1+6,-0.6) {$rb_7$};
\node at (-4.1+6,-0.7) {$rb_9$};
\end{tikzpicture}
\end{center}
\caption{The Anderson--Putnam complex of the Penrose tiling \cite{AP}.}
\label{Penrose_AP}
\end{figure}

The self-similar inverse semigroup is generated by the collection of doubly pointed patches consisting of all single tile patches and adjacent pair patches appearing anywhere in a Penrose tiling. The Anderson--Putnam complex \cite[Section 10.4]{AP}, copied in Figure \ref{Penrose_AP}, neatly illustrates each possible two-tile patch using the edge identifications. For each connected pair of tiles, the doubly pointed patch $[q,X,p]$ represents translation from $p$ to $q$ across a specific edge of tile $p$ where $X \in \{B,L,R\}$ denotes crossing the \emph{B}ottom, \emph{L}eft, or \emph{R}ight edges of prototile $p$ with orientation from Figure \ref{Penrose_proto}.

We now describe the generating two-tile patch elements associated with moving across an edge of $a_i$. Again, we note that all subscripts are treated mod$\ 10$ and $w \in \mc{F}$.

	\begingroup
		\allowdisplaybreaks
\begin{align*}
[ra_{i+5},B,a_i] \cdot (a_{i},a_{i+3})w&=(ra_{i+5},b_{i+9}) \ [b_{i+9},L,a_{i+3}] \cdot w; \\
[b_{i+6},L,a_i] \cdot (a_{i},a_{i+3})w&=(b_{i+6},a_{i+3})w; \\
[ra_{i+1},R,a_i] \cdot (a_{i},a_{i+3})w&=(ra_{i+1},ra_{i+8}) \ [ra_{i+8},B,a_{i+3}] \cdot w; \\
[ra_{i+5},B,a_i] \cdot (a_{i},rb_{i+6})(rb_{i+6},ra_{i+2})w&=(ra_{i+5},ra_{i+2}) \ [ra_{i+2},R,rb_{i+6}] \cdot (rb_{i+6},ra_{i+2})w; \\ 
[ra_{i+5},B,a_i] \cdot (a_{i},rb_{i+6})(rb_{i+6},b_{i+6})w&=(ra_{i+5},ra_{i+2}) \ [ra_{i+2},R,rb_{i+6}] \cdot (rb_{i+6},b_{i+6})w; \\ 
[ra_{i+5},B,a_i] \cdot (a_{i},rb_{i+6})(rb_{i+6},rb_{i})w&=(ra_{i+5},b_{i+9}) \ [b_{i+9},L,rb_{i+6}] \cdot (rb_{i+6},rb_{i})w; \\
[b_{i+6},L,a_i] \cdot (a_{i},rb_{i+6})w&=(b_{i+6},rb_{i+6})w; \\
[rb_{i+2},R,a_i] \cdot (a_{i},rb_{i+6})w&=(rb_{i+2},rb_{i+6})w.
\end{align*}
	\endgroup

Note that reflection acts on tiles by $a_i \leftrightarrow ra_{-i}$, $b_i \leftrightarrow rb_{-i}$ and edge types by $B \leftrightarrow B$, $L \leftrightarrow R$. Therefore, the above relations determine also those for reflections. For example, the last row determines the relation $[b_{i+8},L,ra_i] \cdot (ra_{i},b_{i+4})w = (b_{i+8},b_{i+4})w$, where we apply the above conversions, write $-2 \equiv 8 \mod 10$ etc., and also substitute $-i$ with $i$.

The following relations describe the generating two-tile patch inverse semigroup elements associated with moving across an edge of $b_i$.

	\begingroup
		\allowdisplaybreaks
\begin{align*}
[rb_{i+5},B,b_i] \cdot (b_{i},a_{i+7})(a_{i+7},a_{i})w&=(rb_{i+5},ra_{i+8}) \ [ra_{i+8},R,a_{i+7}] \cdot (a_{i+7},a_{i})w; \\
[rb_{i+5},B,b_i] \cdot (b_{i},a_{i+7})(a_{i+7},rb_{i+3})w&=(rb_{i+5},rb_{i+9}) \ [rb_{i+9},R,a_{i+7}] \cdot (a_{i+7},rb_{i+3})w; \\
[a_{i+4},L,b_i] \cdot (b_{i},a_{i+7})w&=(a_{i+4},a_{i+7})w; \\
[rb_{i+3},R,b_i] \cdot (b_{i},a_{i+7})w&=(rb_{i+3},b_{i+3}) \ [b_{i+3},L,a_{i+7}] \cdot w; \\
[rb_{i+5},B,b_i] \cdot (b_{i},b_{i+6})(b_{i+6},a_{i+3})w&=(rb_{i+5},rb_{i+9}) \ [rb_{i+9},R,b_{i+6}] \cdot (b_{i+6},a_{i+3})w; \\
[rb_{i+5},B,b_i] \cdot (b_{i},b_{i+6})(b_{i+6},b_{i+2})w&=(rb_{i+5},ra_{i+8}) \ [ra_{i+8},R,b_{i+6}] \cdot (b_{i+6},b_{i+2})w; \\
[rb_{i+5},B,b_i] \cdot (b_{i},b_{i+6})(b_{i+6},rb_{i+6})w&=(rb_{i+5},rb_{i+9}) \ [rb_{i+9},R,b_{i+6}] \cdot (b_{i+6},rb_{i+6})w; \\
[rb_{i+7},L,b_i] \cdot (b_{i},b_{i+6})w&=(rb_{i+7},rb_{i+1}) \ [rb_{i+1},B,b_{i+6}] \cdot w; \\
[ra_{i+2},R,b_i] \cdot (b_{i},b_{i+6})w&=(ra_{i+2},b_{i+6})w; \\
[rb_{i+5},B,b_i] \cdot (b_{i},rb_{i})w&=(rb_{i+5},b_{i+5}) \ [b_{i+5},B,rb_{i}] \cdot w; \\
[a_{i+4},L,b_i] \cdot (b_{i},rb_{i})w&=(a_{i+4},rb_{i})w; \\
[rb_{i+3},R,b_i] \cdot (b_{i},rb_{i})(rb_i,ra_{i+3})w&=(rb_{i+3},ra_{i+6}) \ [ra_{i+6},R,rb_{i}] \cdot (rb_i,ra_{i+3})w; \\
[rb_{i+3},R,b_i] \cdot (b_{i},rb_{i})(rb_i,b_{i})w&=(rb_{i+3},ra_{i+6}) \ [ra_{i+6},R,rb_{i}] \cdot (rb_i,b_{i})w; \\
[rb_{i+3},R,b_i] \cdot (b_{i},rb_{i})(rb_i,rb_{i+4})w&=(rb_{i+3},b_{i+3}) \ [b_{i+3},R,rb_{i}] \cdot (rb_i,rb_{i+4})w.  \qed
\end{align*}
\end{example}
	\endgroup

\end{document}